%% file: paper.tex
\documentclass[11pt]{article}

\usepackage[square, numbers]{natbib}

\usepackage[utf8]{inputenc} % allow utf-8 input
\usepackage[T1]{fontenc}    % use 8-bit T1 fonts
\usepackage[colorlinks=true,linkcolor=blue,urlcolor=black,citecolor=blue,breaklinks]{hyperref}
\usepackage{url}            % simple URL typesetting
\usepackage{booktabs}       % professional-quality tables
\usepackage{amsfonts}       % blackboard math symbols
\usepackage{nicefrac}       % compact symbols for 1/2, etc.
\usepackage{microtype}      % microtypography
\usepackage{fullpage}

\usepackage{enumerate} 
\usepackage{graphicx}
\usepackage{algorithm,algcompatible}
\usepackage{mathrsfs}
\usepackage{mathtools}
\usepackage[font={small}]{caption}

\usepackage{amsmath,amsthm,amssymb,colonequals,etoolbox}
\usepackage[normalem]{ulem} % just for strikeout

\newtoggle{draft}
\toggletrue{draft}
\input{commands}

\title{Certainty Equivalence is Efficient for Linear Quadratic Control}
\author{Horia Mania, Stephen Tu, and Benjamin Recht \\
University of California, Berkeley}

\begin{document}

\maketitle

\input{abstract}
\input{intro}

\input{main}

\input{lqg}

\input{riccati_perturbation_main}
\input{related}
\input{conclusion}

%\section*{Acknowledgements}
%
%We thank Elad Hazan and Martin Wainwright, who both independently asked whether or not it was possible to show a fast rate for LQR.
%%
%As part of the RISE lab, HM is generally supported in part by NSF CISE Expeditions Award CCF-1730628, DHS Award HSHQDC-16-3-00083, and gifts from Alibaba, Amazon 
%Web Services, Ant Financial, CapitalOne, Ericsson, GE, Google, Huawei, Intel, IBM, Microsoft,
%Scotiabank, Splunk and VMware.
%%
%ST is supported by a Google PhD fellowship.
%%
%BR is generously supported in part by ONR awards N00014-17-1-2191, N00014-17-1-2401, and N00014-18-1-2833, the DARPA Assured Autonomy (FA8750-18-C-0101) and Lagrange (W911NF-16-1-0552) programs, a Siemens Futuremakers Fellowship, and an Amazon AWS AI Research Award.
% 

\bibliographystyle{abbrvnat}
\bibliography{lqr}

\clearpage  
\appendix

\input{appendix}

\input{proof_op_bound}

\input{proofs_lqg}   
 
\end{document}

%% file: commands.tex
\newtheorem{myprop}{Proposition}
\newtheorem{mylemma}{Lemma}
\newtheorem{mycorollary}{Corollary}

\newtheorem{mythm}{Theorem}
\newtheorem{myassumption}{Assumption}
\DeclareMathOperator*{\Tr}{\mathrm{tr}}

\newcommand{\Ccal}{\mathcal{C}}

\newcommand{\R}{\ensuremath{\mathbb{R}}}
\newcommand{\RR}{\mathbb{R}}
\newcommand{\EE}{\mathbb{E}}

\newcommand{\norm}[1]{\lVert #1 \rVert}
\newcommand{\bignorm}[1]{\left\lVert #1 \right\rVert}

\newcommand{\E}{\mathbb{E}}
\newcommand{\abs}[1]{\ensuremath{| #1 |}}

\renewcommand{\vec}{\mathrm{vec}}

\newcommand{\T}{\mathsf{T}}

\newcommand{\smin}{\underline{\sigma}}

\newcommand{\calN}{\mathcal{N}}

\newcommand{\calH}{\mathcal{H}}

\newcommand{\calS}{\mathcal{S}}
\newcommand{\calO}{\mathcal{O}}
\newcommand{\Ocal}{\mathcal{O}}
\newcommand{\Scal}{\mathcal{S}}
\newcommand{\Ncal}{\mathcal{N}}
\newcommand{\Tcal}{\mathcal{T}}
\newcommand{\opT}{\operatorname{\mathcal{T}}}
\newcommand{\opH}{\operatorname{\mathcal{H}}}
\newcommand{\opG}{\operatorname{\mathcal{G}}}
\newcommand{\opPhi}{\operatorname{\Phi}}

\newcommand{\Otilde}{\widetilde{\calO}}

\newcommand{\cvectwo}[2]{\begin{bmatrix} #1 \\ #2 \end{bmatrix}}

\newcommand{\bmattwo}[4]{\begin{bmatrix} #1 & #2 \\ #3 & #4 \end{bmatrix}}

\newcommand{\Ah}{\widehat{A}}
\newcommand{\Bh}{\widehat{B}}
\newcommand{\Ch}{\widehat{C}}
\newcommand{\Lh}{\widehat{L}}
\newcommand{\Kh}{\widehat{K}}
\newcommand{\Mh}{\widehat{M}}
\newcommand{\Nh}{\widehat{N}}
\newcommand{\Qh}{\widehat{Q}}
\newcommand{\Xih}{\widehat{\Xi}}
\newcommand{\Sigmah}{\widehat{\Sigma}}

\newcommand{\Ahat}{\widehat{A}}
\newcommand{\Bhat}{\widehat{B}}

\newcommand{\Astar}{A_{\star}}
\newcommand{\Bstar}{B_{\star}}
\newcommand{\Cstar}{C_{\star}}
\newcommand{\Jstar}{J_{\star}}
\newcommand{\Sstar}{S_{\star}}
\newcommand{\Lstar}{L_{\star}}
\newcommand{\Jhat}{\widehat{J}}
\newcommand{\Shat}{\widehat{S}}

\newcommand{\dlyap}{\mathsf{dlyap}}
\newcommand{\Avg}{\mathsf{Avg}}
\newcommand{\statedim}{n}
\newcommand{\inputdim}{d}

\newcommand{\Ph}{\widehat{P}}
\newcommand{\Phat}{\widehat{P}}
\newcommand{\Khat}{\widehat{K}}
\newcommand{\Pstar}{P_\star}

\newcommand{\vece}{\mathbf{e}}
\newcommand{\vecx}{\mathbf{x}}
\newcommand{\vecu}{\mathbf{u}}
\newcommand{\vecy}{\mathbf{y}}
\newcommand{\vecw}{\mathbf{w}}
\newcommand{\vecv}{\mathbf{v}}

\newcommand{\hatvecx}{\widehat{\vecx}}

\newcommand{\tM}{\tilde{M}}
\newcommand{\tA}{\tilde{A}}
\newcommand{\tB}{\tilde{B}}
\newcommand{\tC}{\tilde{C}}
\newcommand{\tK}{\tilde{K}}
\newcommand{\tL}{\tilde{L}}

\newcommand{\trho}{\tilde{\rho}}
\newcommand{\tTheta}{\tilde{\Theta}}

\newcommand{\tXi}{\tilde{\Xi}}

\newcommand{\Kstar}{K_\star}

\newcommand{\lbsmin}{\nu}
\newcommand{\transient}[2]{\tau(#1, #2)}

%% file: abstract.tex
%!TEX root = paper.tex

\begin{abstract}
We study the performance of the \emph{certainty equivalent controller}
on Linear Quadratic (LQ) control problems with unknown transition dynamics.
%where estimates of the dynamics are directly used
%to design a controller without taking into account any modelling error.
%
We show that for both the fully and partially observed settings,
the sub-optimality gap between the cost incurred
by playing the certainty equivalent controller on the true system and the cost incurred by using the optimal LQ controller
enjoys a fast statistical rate, scaling as the \emph{square} of the parameter error.
To the best of our knowledge, our result is the first sub-optimality guarantee
in the partially observed Linear Quadratic Gaussian (LQG) setting.
Furthermore, in the fully observed Linear Quadratic Regulator (LQR),
our result improves upon recent work by \citet{dean2017sample}, who present an algorithm achieving a sub-optimality gap linear in the parameter error.
%
%
%This is in contrast to recent work by , who show that
%a robust synthesis procedure that explicitly takes model uncertainty into account
%achieves a sub-optimality gap \emph{linear} in the parameter error.
%
A key part of our analysis relies on perturbation bounds for discrete Riccati equations.
{We provide two new perturbation bounds, one that expands on
an existing result from \citet{konstantinov1993perturbation}, and
another based on a new elementary proof strategy.}
%
% {\color{red}
%  Our results can be used to prove a $\Otilde(\sqrt{T})$ regret guarantee with explicit constants for the adaptive LQR setting}
%
\end{abstract}

%The growing scope of applied control theory and reinforcement learning demands a better understanding
%of the amounts of data required for controlling dynamical systems. We study the sample complexity of the linear quadratic regulator with unknown dynamics. We show that the nominal controller, synthesized based on estimated parameters, achieves a fast statistical rate for optimizing the cost objective. To prove this result we offer a new way to upper bound changes in the solutions to Riccati equations.

%% file: intro.tex
%!TEX root = paper.tex

\section{Introduction}

One of the most straightforward methods for controlling a dynamical system with
unknown transitions is based on the \emph{certainty equivalence principle}: a model of the system is fit
by observing its time evolution, and a control policy is then designed by
treating the fitted model as the truth \cite{aastrom2013adaptive}. Despite the simplicity of this method,
it is challenging to guarantee its efficiency because small
modeling errors may propagate to large, undesirable behaviors on long time horizons.
As a result, most work on controlling systems with unknown dynamics has
explicitly incorporated robustness against model uncertainty~\cite{dean2017sample,dean18regret, iyengar2005robust, nilim2005robust, xu2010distributionally, ZDGBook}.

In this work, we show that for the standard baseline of controlling an unknown
linear dynamical system with a quadratic objective function known as Linear Quadratic (LQ) control, certainty equivalent
control synthesis achieves \emph{better} cost than prior methods that account
for model uncertainty.
Our results hold for both the fully observed Linear Quadratic Regulator (LQR)
and the partially observed Linear Quadratic Gaussian (LQG) setting.
For offline control, where one collects some
data and then designs a fixed control policy to be run on an infinite time
horizon, we show that the gap between the performance of the certainty equivalent controller
and the optimal control policy scales \emph{quadratically} with the error in the model
parameters for both LQR and LQG. 
To the best of our knowledge, we provide the first sub-optimality guarantee for LQG.
Moreover, in the LQR setting our work improves upon the recent result of \citet{dean2017sample}, who
present an algorithm that achieves a sub-optimality gap linear in the parameter
error. 
In the case of online LQR control, where one adaptively improves the control
policy as new data comes in, our offline result implies that a simple,
polynomial time algorithm using $\varepsilon$-greedy exploration suffices for
nearly optimal $\Otilde(\sqrt{T})$ regret.
%
% Prior to our work, existing algorithms for adaptive LQR either
% require sub-routines for which efficient algorithms are not known~\cite{abbasi2011regret, faradonbeh18, ibrahimi12},
% achieve sub-optimal regret~\cite{abbasi18, abeille17, dean18regret}, or apply only to scalar systems \cite{abeille18}.
% {\color{red} \citet{faradonbeh18} also argue that online certainty equivalent control achieves $\Otilde(\sqrt{T})$ regret; we further discuss their result in Section~\ref{sec:related}.}
   
This paper is structured as follows. In Section~\ref{sec:main} we present backround concepts and discuss our main result for LQR; we compare it to prior guarantees and discuss its consequences. Our results rely on a study of the sensitivity to parameter perturbations of the Bellman equation of LQR, known as the discrete
algebraic Riccati equation. 
In Section~\ref{sec:proof_main}, we assume the existence of a sensitivity guarantee, and use the guarantee to prove a meta theorem which quantifies the performance of the certainty equivalent controller.
Then, in Section~\ref{sec:lqg} we extend our main result for LQR to the more general case of LQG.
Section~\ref{sec:riccati} contains two explicit and interpretable upper bounds on the sensitivity of the Riccati solution: one based on a proof strategy proposed by \citet{konstantinov1993perturbation} and one based on a direct approach that is of independent interest. 
We conclude and discuss future directions in Section~\ref{sec:conclusion}.

%% file: main.tex
%!TEX root = paper.tex

\section{Main Results for the Linear Quadratic Regulator}
\label{sec:main}

An instance of the linear quadratic regulator (LQR) is defined by four matrices: two matrices $\Astar \in \RR^{\statedim \times \statedim}$ and $\Bstar \in \RR^{\statedim \times \inputdim}$ that define the linear dynamics and two positive semidefinite matrices $Q \in \RR^{\statedim \times \statedim}$ and $R \in \RR^{\inputdim \times \inputdim}$ that define the cost function. Given these matrices, the goal of LQR is to solve the optimization problem
\begin{align}
  \label{eq:lqr}
  \min_{\vecu_0, \vecu_1, \ldots} &\lim_{T \to \infty} \EE \left[ \frac{1}{T} \sum_{t = 0}^T \vecx_t^\top Q \vecx_t + \vecu_t^\top R \vecu_t \right]
  \; \text{s.t. } \; \vecx_{t + 1} = \Astar \vecx_t + \Bstar \vecu_t + \vecw_t,
\end{align}
where $\vecx_t$, $\vecu_t$ and $\vecw_t$ denote the state, input (or action), and noise at time $t$, respectively. The expectation is over the initial state $\vecx_0 \sim \Ncal(0, I_\statedim)$ and the i.i.d.\ noise $\vecw_t \sim \Ncal(0, \sigma_w^2 I_\statedim)$. 
%The state and noise vectors are $\statedim$ and $\inputdim$ dimensional, respectively.  The input at time $t$ is allowed to depend on the state at time $t$ and all the previous states and actions. 
%Nonetheless, 
When the problem parameters $(\Astar, \Bstar, Q, R)$ are known the optimal policy is given by linear feedback, $\vecu_t = \Kstar \vecx_t$, where $\Kstar = - (R + \Bstar^\top \Pstar \Bstar)^{-1}\Bstar^\top \Pstar \Astar$ where $\Pstar$ is the (positive definite) solution to the discrete Riccati equation
\begin{align}
\label{eq:riccati}
  \Pstar = \Astar^\top \Pstar \Astar - \Astar^\top \Pstar \Bstar (R + \Bstar^\top \Pstar \Bstar)^{-1} \Bstar^\top \Pstar \Astar + Q
\end{align}
and can be computed efficiently \cite[see e.g.]{anderson07}. In the sequel we use the notation $ \mathsf{dare}(\Astar, \Bstar, Q, R)$ to denote the unique positive semidefinite solution of \eqref{eq:riccati}. Problem~\eqref{eq:lqr} considers an average cost over an infinite horizon. The optimal controller for the finite horizon variant is also static and linear, but time-varying. The LQR solution in this case can be computed efficiently via dynamic programming.

In this work we are interested in the control of a linear dynamical system with unknown transition parameters $(\Astar, \Bstar)$ based on estimates  $(\Ahat, \Bhat)$. The cost matrices $Q$ and $R$ are assumed known. We analyze the \emph{certainty equivalence approach}: use the estimates $(\Ahat, \Bhat)$ to solve the optimization problem \eqref{eq:lqr} while disregarding the modeling error, and use the resulting controller on the true system $(\Astar, \Bstar)$.
We interchangeably refer to the resulting policy as the \emph{certainty equivalent controller} or, following \citet{dean2017sample}, the \emph{nominal controller}.
We denote by $\Phat$ the solution to the Riccati equation~\eqref{eq:riccati} associated with the parameters $(\Ahat, \Bhat)$ and let $\Khat$ be the corresponding controller. We denote by $J(A, B, K)$ the cost \eqref{eq:lqr} obtained by using the actions $\vecu_t = K \vecx_t$ on the system $(A,B)$, and  we use $\Jhat$ and $\Jstar$ to denote $J(\Astar, \Bstar, \Khat)$ and $J(\Astar, \Bstar, \Kstar)$, respectively.

Let $\varepsilon \geq 0$ such that $\norm{\Astar - \Ahat} \leq \varepsilon$ and $\norm{\Bstar - \Bhat} \leq \varepsilon$.
(Here and throughout this work we use $\norm{\cdot}$ to denote the Euclidean norm for
vectors as well as the spectral (operator) norm for matrices.)
\citet{dean2017sample} introduced a robust controller that achieves $\Jhat - \Jstar \leq C_1(\Astar, \Bstar, Q, R)\varepsilon$ for some complexity term $C_1(\Astar,\Bstar, Q, R)$ that depends on the problem parameters. We show that the nominal controller $\vecu_t = \Khat \vecx_t$ achieves $\Jhat - J_\star \leq C_2(\Astar, \Bstar, Q, R) \varepsilon^2$.
Both results require $\varepsilon$ to be sufficiently small (as a function of the problem parameters) and it is important
to note that $\varepsilon$ must be much smaller for the nominal controller to
be guaranteed to stabilize the system than for the robust controller proposed
by \citet{dean2017sample}. However, our result shows that once the
estimation error $\varepsilon$ is small enough, the nominal controller performs
better: the sub-optimality gap scales as $\calO(\varepsilon^2)$ versus
$\calO(\varepsilon)$.
Both the more stringent requirement on $\varepsilon$ and better performance
of nominal control compared to robust control, when the estimation error is sufficiently small, were observed empirically  by \citet{dean2017sample}.

Before we can formally state our result we need to introduce a few more concepts and assumptions. It is common to assume that the cost matrices $Q$ and $R$ are positive definite. Under an additional observability assumption, this condition can be relaxed to $Q$ being positive semidefinite.  
\begin{myassumption}
\label{ass:strg_cvx}
The cost matrices $Q$ and $R$ are positive definite. Since scaling both $Q$ and $R$ does not change the optimal controller $\Kstar$, we can assume without loss of generality that $\smin(R) \geq 1$.
\end{myassumption}
A square matrix $M$ is \emph{stable} if its spectral radius $\rho(M)$ is (strictly) smaller than one.
Recall that the spectral radius is defined as $\rho(M) = \max \{ \abs{\lambda} : \lambda \text{ is an eigenvalue of } M \}$.
 A linear dynamical system $(A,B)$ in feedback with $K$ is fully described by the \emph{closed loop matrix} $A + BK$. More precisely, in this case $\vecx_{t + 1} = (A + BK) \vecx_t + \vecw_t$. For a static linear controller $\vecu_t = K \vecx_t$ to achieve finite LQR cost it is necessary and sufficient that the closed loop matrix is stable.

In order to quantify the growth or decay of powers of a square matrix $M$, we define
 \begin{align}
   \label{eq:transient}
\transient{M}{ \rho} := \sup \left\{\norm{M^k} \rho^{-k} \colon k\geq 0 \right\}.
 \end{align}
 In other words, $\transient{M}{\rho}$ is the smallest value such that $\norm{M^k} \leq \transient{M}{\rho} \rho^k$ for all $k \geq 0$.
 We note that $\transient{M}{\rho}$ might be infinite, depending on the value of $\rho$, and it is always greater or equal than one. If $\rho$ is larger than $\rho(M)$, we are guaranteed to have a finite $\transient{M}{\rho}$ (this is a consequence of Gelfand's formula). In particular, if $M$ is a stable matrix, we can choose $\rho < 1$ such that $\transient{M}{\rho}$ is finite. Also, we note that $\transient{M}{\rho}$ is a decreasing function of $\rho$; if $\rho \geq \norm{M}$, we have $\transient{M}{\rho} = 1$.
At a high level, the quantity $\transient{M}{\rho}$ measures the degree of transient response of the linear system
$\vecx_{t+1} = M \vecx_t + \vecw_t$. In particular, when $M$ is stable, $\transient{M}{\rho}$ can be upper bounded by the
$\mathcal{H}_\infty$-norm of the system defined by $M$, which is the $\ell_2$ to $ \ell_2$ operator norm of the system and a fundamental quantity in robust control \cite[see][for more details]{tu17}.

Throughout this work we use the quantities $\Gamma_\star := 1 + \max\{\norm{\Astar}, \norm{\Bstar}, \norm{\Pstar}, \norm{K_\star}\}$ and $L_\star := \Astar + \Bstar \Kstar$. We use $\Gamma_\star$ as a uniform upper bound on the spectral norms of the relevant matrices for the sake of algebraic simplicity. We are ready to state our meta theorem.

\begin{mythm}
\label{thm:meta}
Suppose $\inputdim \leq \statedim$.
Let $\gamma > 0$ such that $\rho(L_\star) \leq \gamma < 1$. Also, let $\varepsilon > 0$ such that $\norm{\Ahat - \Astar} \leq \varepsilon$ and $\norm{\Bhat - \Bstar} \leq \varepsilon$ and assume $\norm{\Phat - \Pstar} \leq f(\varepsilon)$  for some function $f$ such that $f(\varepsilon) \geq \varepsilon$. Then, under Assumption~\ref{ass:strg_cvx} the certainty equivalent controller $\vecu_t = \Khat \vecx_t$ achieves
\begin{align}
  \Jhat - \Jstar \leq 200 \, \sigma_w^2 \,  \inputdim \, \Gamma_\star^9 \, \frac{\transient{\Lstar}{\gamma}^2}{ 1 - \gamma^2} f(\varepsilon)^2,
\end{align}
as long as $f(\varepsilon)$ is small enough so that the right hand side is smaller than $\sigma_w^2$.
\end{mythm}

In Section~\ref{sec:riccati} we present two upper bounds $f(\varepsilon)$ on $\norm{\Phat - \Pstar}$: one based on a proof technique proposed by \citet{konstantinov1993perturbation} and one based on our direct approach. Both of these upper bounds satisfy $f(\varepsilon) = \calO(\varepsilon)$ for $\varepsilon$ sufficiently small. For simplicity, in this section we only specialize our meta-theorem (Theorem~\ref{thm:meta}) using the perturbation result from our direct approach.

To state a specialization of Theorem~\ref{thm:meta} we need a few more concepts. A linear system $(A,B)$ is called \emph{controllable} when the \emph{controllability matrix}
  %\begin{align*}
    $\begin{bmatrix} 
      B & AB & A^2 B & \ldots & A^{\statedim - 1} B
     \end{bmatrix}$
  %\end{align*}
  has full row rank. Controllability is a fundamental concept in control
  theory; it states that there exists a sequence of inputs to the system
  $(A,B)$ that moves it from any starting state to any final state in at most
  $\statedim$ steps. In this work we quantify how controllable a linear system
  is. We denote, for any integer $\ell \geq 1$, the matrix $\Ccal_\ell :=
  \begin{bmatrix} B & A B & \ldots & A^{\ell - 1} B \end{bmatrix}$ and call the
  system \emph{ $(\ell, \lbsmin)$-controllable } if the $\statedim$-th
  singular value of $\Ccal_\ell$ is greater or equal than $\lbsmin$, i.e.
  $\smin(\Ccal_\ell) = \sqrt{\lambda_{\min}\left(\Ccal_\ell
  \Ccal_\ell^\top\right)} \geq \lbsmin$. Intuitively, the larger $\lbsmin$ is,
  the less control effort is needed to move the system between two
  different states.
 \begin{myassumption}
    \label{ass:controllable}
   We assume the unknown system $(\Astar, \Bstar)$ is $(\ell, \lbsmin)$-controllable, with $\lbsmin > 0$.
 \end{myassumption} 
Assumption \ref{ass:controllable} was used in a different context by \citet{cohen18}. For any controllable system and any $\ell \geq \statedim$ there exists $\lbsmin > 0$ such that the system is $(\ell, \lbsmin)$-controllable. {Therefore, $(\ell, \lbsmin)$-controllability is really not much stronger of an assumption than controllability.}
As $\ell$ grows minimum singular value $\smin(\Ccal_\ell)$ also grows and therefore a larger $\lbsmin$ can be chosen so that the system is still $(\ell, \lbsmin)$ controllable.

Note that controllability is not necessary for LQR to have a well-defined solution: the weaker requirement is that of
\emph{stabilizability}, in which there exists a feedback matrix $K$ so that $\Astar + \Bstar K$ is stable. The result of \citet{dean2017sample} only requires stabilizability. While our upper bound on $\norm{\Phat- \Pstar}$ requires controllability, the result of \citet{konstantinov1993perturbation} only requires stabilizability. However, our upper bound on $\norm{\Phat - \Pstar}$ is sharper for some classes of systems (see Section~\ref{sec:riccati}). Together with Theorem~\ref{thm:meta}, our perturbation result, presented in Section~\ref{sec:riccati}, yields the following guarantee. 
\begin{mythm}
\label{thm:fast_rate}
Suppose that $\inputdim \leq \statedim$.
Let $\rho$ and $\gamma$ be two real values such that $\rho(\Astar) \leq \rho$ and $\rho(L_\star) \leq \gamma < 1$.
Also, let $\varepsilon > 0$ such that $\norm{\Ahat - \Astar} \leq \varepsilon$ and $\norm{\Bhat - \Bstar} \leq \varepsilon$ and define $\beta = \max\{1, \varepsilon  \transient{\Astar}{\rho}  + \rho\}$. Under Assumptions~\ref{ass:strg_cvx} and \ref{ass:controllable}, the certainty equivalent controller $\vecu_t = \Khat \vecx_t$ satisfies the suboptimality gap
\begin{align}
    \Jhat - \Jstar \leq \Ocal(1)\,  \sigma^2_w \, \inputdim \, \ell^5 \, \Gamma_\star^{15} \, \transient{\Astar}{\rho}^6 \beta^{4(\ell - 1)} \, \frac{\transient{L_\star }{\gamma}^2}{1 - \gamma^2} \frac{\max\{\norm{Q}^2, \norm{R}^2 \}}{\min \left\{ \smin(Q)^2, \smin(R)^2 \right\}} \left(1 + \frac{1}{\lbsmin}\right)^2 \varepsilon^2\:, \label{eq:fast_rate}
\end{align}
as long as the right hand side is smaller than $\sigma_w^2$. Here, $\Ocal(1)$ denotes a universal constant.
\end{mythm}
The exact form of Equation~\ref{eq:fast_rate}, such as the polynomial
dependence on $\ell$, $\Gamma_\star$, etc, can be improved at the expense of conciseness of the
expression. In our proof we optimized for the latter.
The factor $\max\{\norm{Q}^2, \norm{R}^2 \} / \min \left\{ \smin(Q)^2, \smin(R)^2 \right\}$ is the squared condition number of the cost function, a natural quantity in the context of the optimization problem \eqref{eq:lqr}, which can be seen as an infinite dimensional quadratic program with a linear constraint.
  The term $\frac{\transient{L_\star }{\gamma}^2}{1 - \gamma^2}$ quantifies the rate at which the optimal controller drives the state towards zero.
Generally speaking, the less stable the optimal closed loop system is, the larger this term becomes.

An interesting trade-off arises between the factor $\ell^5 \beta^{4(\ell -
1)}$ (which arises from upper bounding perturbations of powers of $\Astar$
on a time interval of length $\ell$) and the factor $\lbsmin$ (the lower
bound on $\smin(\Ccal_\ell)$), which is increasing in $\ell$. Hence,
the parameter $\ell$ should be seen as a free-parameter that can be tuned to
minimize the right hand side of \eqref{eq:fast_rate}.
Now, we  specialize Theorem~\ref{thm:fast_rate} to a few cases.

\paragraph{Case: $\Astar$ is contractive, i.e. $\norm{\Astar} < 1$.}

In this case, we can choose $\rho = \norm{\Astar}$ and $\varepsilon$ small enough so that $\varepsilon \leq 1 - \norm{\Astar}$.
Then, \eqref{eq:fast_rate} simplifies to:
\begin{align*}
    \Jhat - \Jstar \leq \Ocal(1)\, \inputdim \, \sigma^2_w \, \ell^5 \, \Gamma_\star^{15} \frac{\transient{L_\star}{\gamma}^2}{1 - \gamma^2} \frac{\max\{\norm{Q}^2, \norm{R}^2 \}}{\min \left\{ \smin(Q)^2, \smin(R)^2 \right\}} \left(1 + \frac{1}{\lbsmin}\right)^2 \varepsilon^2 \:.
\end{align*}

\paragraph{Case: $\Bstar$ has rank $\statedim$ .}

In this case, we can choose $\ell=1$.
Then, \eqref{eq:fast_rate} simplifies to:
\begin{align*}
    \Jhat - \Jstar \leq \Ocal(1)\, \inputdim \, \sigma^2_w \, \Gamma_\star^{15}\transient{\Astar}{\rho}^6  \frac{\transient{L_\star}{\gamma}^2}{1 - \gamma^2}  \frac{\max\{\norm{Q}^2, \norm{R}^2 \}}{\min \left\{ \smin(Q)^2, \smin(R)^2 \right\}} \left(1 + \frac{1}{\lbsmin}\right)^2 \varepsilon^2 \:.
\end{align*}

\subsection{Comparison to Theorem 4.1 of \texorpdfstring{\citet{dean2017sample}}{Dean et al.}.}

\citet{dean2017sample} show that when their robust synthesis procedure is run
with estimates $(\Ahat, \Bhat)$ satisfying $\max\{\norm{\Ahat - \Astar},\norm{\Bhat - \Bstar}\} \leq \varepsilon \leq \left[ {5(1+\norm{\Kstar})\Psi_\star} \right]^{-1}$, the resulting controller satisfies:
\begin{align}
    \Jhat - \Jstar \leq 10(1+\norm{\Kstar})\Psi_\star \Jstar \varepsilon + \calO(\varepsilon^2) \:. \label{eq:orig_lqr_bound_squared}
\end{align}
% \begin{align}
%     \sqrt{\Jhat} - \sqrt{\Jstar} \leq 5 (1+\norm{\Kstar}) \Psi_\star \sqrt{\Jstar} \varepsilon \:. \label{eq:orig_lqr_bound}
% \end{align}
Here, the quantity $\Psi_\star := \sup_{z \in \mathbb{T}} \norm{ (zI_\statedim - L_\star)^{-1}}$ is the $\mathcal{H}_\infty$-norm of the
optimal closed loop system $L_\star$.
% Equation~\ref{eq:orig_lqr_bound} implies that:
% \begin{align}
%     \Jhat - \Jstar \leq 10(1+\norm{\Kstar})\Psi_\star \Jstar \varepsilon + \calO(\varepsilon^2) \:. \label{eq:orig_lqr_bound_squared}
% \end{align}
In order to compare Equation~\ref{eq:orig_lqr_bound_squared} to Equation~\ref{eq:fast_rate},
we upper bound the quantity $\Psi_\star$ in terms of $\transient{L_\star}{\gamma}$ and $\gamma$.
In particular, by a infinite series expansion of the inverse $(zI_\statedim - L_\star)^{-1}$ we can show $\Psi_\star \leq \frac{\transient{L_\star}{\gamma}}{1-\gamma}$. 
%\begin{align*}
%    \Psi_\star &= \sup_{z \in \mathbb{T}} \norm{(zI_\statedim - L_\star)^{-1}}  = \sup_{z \in \mathbb{T}} \bignorm{\sum_{k=0}^{\infty} L_\s%tar^k z^{-(k+1)}} \leq \sum_{k=0}^{\infty} \norm{ L_\star^k } \leq \frac{\transient{L_\star}{\gamma}}{1-\gamma} \:.
%\end{align*}
Also, we have $\Jstar = \sigma_w^2 \Tr(\Pstar) \leq \sigma_w^2 \statedim \Gamma_\star$.
Therefore, Equation~\ref{eq:orig_lqr_bound_squared} gives us that:
\begin{align*}
    \Jhat - \Jstar \leq \calO(1) \statedim \sigma_w^2 \Gamma_\star^2  \frac{\transient{L_\star}{\gamma}}{1-\gamma} \varepsilon + \calO(\varepsilon^2) \:.
\end{align*}
We see that the dependence on the parameters
$\Gamma_\star$ and $\transient{L_\star}{\gamma}$ is significantly milder compared to Equation~\ref{eq:fast_rate}.
Furthermore, this upper bound is valid for larger $\varepsilon$ than the upper bound given in Theorem~\ref{thm:fast_rate}.
Comparing these upper bound suggests that there is a price to pay for obtaining a fast rate, and
that in regimes of moderate uncertainty (moderate size of $\varepsilon$),
being robust to model uncertainty is important. This observation is supported by the empirical
results of \citet{dean2017sample}.

A similar trade-off between slow and fast rates arises in the setting of
first-order convex stochastic optimization.
The convergence rate $\calO(1/\sqrt{T})$ of the stochastic gradient descent method can be improved
to $\calO(1/T)$ under a strong convexity assumption.
However, the performance of stochastic gradient descent, which can achieve a $\calO(1/T)$ rate,
is sensitive to poorly estimated problem parameters \cite{Nemirovski09}. Similarly, in the case of LQR,
the nominal controller achieves a fast rate, but it is much more sensitive to estimation error than the robust controller of \citet{dean2017sample}.

\paragraph{End-to-end guarantees.}

Theorem~\ref{thm:fast_rate} can be combined with finite sample learning guarantees
(e.g.~\cite{dean2017sample, faradonbeh17a, sarkar2018fast, simchowitz18}) to obtain
an end-to-end guarantee similar to Proposition 1.2 of \citet{dean2017sample}.
In general, estimating the transition parameters from $N$ samples yields an estimation error that scales as $\calO(1/\sqrt{N})$. Therefore, Theorem~\ref{thm:fast_rate} implies that
$\Jhat-\Jstar \leq \calO(1/N)$ instead of the $\Jhat-\Jstar \leq \calO(1/\sqrt{N})$ rate
from Proposition 1.2 of \citet{dean2017sample}.
This is similar to the case of linear regression, where $\calO(1/\sqrt{N})$ estimation error for the parameters translates to a $\calO(1/N)$ \emph{fast rate} for prediction error.
Furthermore, \citet{simchowitz18} and \citet{sarkar2018fast} showed that faster
estimation rates are possible for some linear dynamical systems.
Theorem~\ref{thm:fast_rate} translates such rates into control suboptimality
guarantees in a transparent way.

Our result explains the behavior observed in Figure~4 of \citet{dean2017sample}. The authors propose two procedures for synthesizing robust controllers for LQR with unknown transitions: one which guarantees robustness of the performance gap $\Jhat - \Jstar$, and one which only guarantees the stability of the closed loop system. \citet{dean2017sample} observed that the latter performs better in the small estimation error regime, which happens because the robustness constraint of the synthesis procedure becomes inactive when the estimation error is small enough. Then, the second robust synthesis procedure effectively outputs the certainty equivalent controller, which we now know to achieve a fast rate.

%performs better than optimizing a robust cost upper bound (subject to the same robustness constraints).
%This is because as the system uncertainty diminishes and the robustness constraints become inactive, the former effectively reduces
%to nominal control and hence exhibits a fast rate.}

\subsection{Nearly optimal \texorpdfstring{$\Otilde(\sqrt{T})$}{} regret in the adaptive setting}
\label{sec:online}

The regret formulation of adaptive LQR was first proposed by \citet{abbasi2011regret}.
The task is to design an adaptive algorithm $\{ \vecu_t \}_{t \geq 0}$
to minimize regret, as defined by $\mathsf{Regret}(T) := \sum_{t=1}^{T} \vecx_t^\top Q \vecx_t + \vecu_t^\top R \vecu_t - T J_\star$.
\citet{abbasi2011regret} study the performance of optimism in the face of uncertainty (OFU)
and show that it has $\Otilde(\sqrt{T})$ regret, which is nearly optimal for this problem formulation.
However, the OFU algorithm requires repeated solutions to a non-convex optimization problem
for which no known efficient algorithm exists.

% However, the OFU algorithm repeatedly requires solutions to the following optimization problem:
% \begin{align*}
%   (\widetilde{A}, \widetilde{B}) = \arg\min_{(A,B) \in \Scal} J(A, B) \:.
% \end{align*}
% $J(A, B)$ is shorthand for the cost of the optimal controller on the LQR instance
% $(A, B, Q, R)$ and $\Scal$ denotes a confidence set.
% Even if $\Scal$ is a convex set, this optimization problem is non-convex and no provably efficient
% algorithm is currently known.
%\footnote{In practice, it is observed that projected gradient descent performs well.}

To deal with the computational issues of OFU, \citet{dean18regret} propose to
analyze the behavior of $\varepsilon$-greedy exploration using
the suboptimality gap results of \citet{dean2017sample}.
In the context of continuous control, $\varepsilon$-greedy exploration refers
to the application of the control law
$\vecu_t = \pi(\vecx_t, \vecx_{t-1}, ..., \vecx_0) + \eta_t$
with $\eta_t \sim \calN(0, \sigma_{\eta,t}^2 I_\inputdim)$,
where $\pi$ is the policy, updated in epochs, and $\sigma_{\eta,t}^2$ is the variance of the exploration noise.
\citet{dean18regret} set the variance of the exploration noise as $\sigma_{\eta,t}^2 \sim t^{-1/3}$, and show that their method achieves $\Otilde(T^{2/3})$ regret.
They use epochs of size $2^{i}$ and decompose the regret roughly as
$\mathsf{Regret}(T) = \calO\left( T(\Jhat - \Jstar) + T \sigma_{\eta,T}^2 \right)$.
Since the estimation error of the model parameters scales as $\calO((\sigma_{\eta,T} \sqrt{T})^{-1})$,
and since the suboptimality gap $\Jhat - \Jstar$ of the robust controller is linear in the estimation error, we have
$\mathsf{Regret}(T) = \calO\left( \frac{\sqrt{T}}{\sigma_{\eta,T}} + T \sigma_{\eta,T}^2 \right)$.
Then, setting $\sigma_{\eta,t}^2 \sim t^{-1/3}$ balances these two terms and yields $\Otilde(T^{2/3})$ regret.
However, Theorem~\ref{thm:fast_rate}, which states that the gap $\Jhat - \Jstar$ for the nominal controller depends quadratically on the estimation rate, implies that online certainty equivalent control achieves
$\mathsf{Regret}(T) = \calO\left( \frac{1}{\sigma_{\eta,T}^2} + T \sigma_{\eta,T}^2 \right)$.
Here, the optimal variance of the exploration noise scales as $\sigma_{\eta,t}^2 \sim t^{-1/2}$, yielding $\Otilde(\sqrt{T})$ regret.
We note that the observation that certainty equivalence coupled with $\varepsilon$-greedy 
exploration achieves $\Otilde(\sqrt{T})$ regret was first made by \citet{faradonbeh18}.

%rely on the robust control procedure from \cite{dean2017sample},
%the $\calO(\varepsilon)$ suboptimality gap forces the scale of exploration noise
%to be set as $\sigma_{\eta,t}^2 \sim T^{-1/3}$, which results in the suboptimal regret
%Theorem~\ref{thm:fast_rate} allows us to use nominal control instead to synthesize $\pi$,
%and the $\calO(\varepsilon^2)$ suboptimality gap allows us to more aggressively set $\sigma_{\eta,t}^2 \sim T^{-1/2}$,
%resulting in an $\Otilde(\sqrt{T})$ regret.  Combining these modifications with the same proof strategy
%of \cite{dean18regret} yields the following result.

\begin{mycorollary}
\label{cor:sqrt_T}
(Informal) $\varepsilon$-greedy exploration with exploration schedule
$\sigma_{\eta,t}^2 \sim t^{-1/2}$ combined with certainty equivalent control
yields an adaptive LQR algorithm with regret bounded as $\Otilde(\sqrt{T})$.
\end{mycorollary}

\subsection{Proof of  Theorem~\ref{thm:meta}}
\label{sec:proof_main}

In this section we prove our meta theorem; we show how an upper bound $\norm{\Phat - \Pstar} \leq f(\varepsilon)$ can be used to quantify the mismatch between the performance of the the nominal controller and the optimal controller. First, we upper bound $\norm{\Khat - \Kstar}$ and offer a condition on this mismatch size so that  $\Astar + \Bstar \Khat$ is a stable matrix. 
The next two optimization results are helpful in proving $\norm{\Khat - \Kstar}$ is small.
\begin{mylemma}
  \label{lem:strongly_convex_min_perturb}
  Let $f_1, f_2$ be two $\mu$-strongly convex twice differentiable
  functions.  Let $\vecx_1 = \arg\min_\vecx f_1(\vecx)$ and
  $\vecx_2 = \arg\min_\vecx f_2(\vecx)$.  Suppose  $\norm{ \nabla f_1(\vecx_2)} \leq \varepsilon$, then
  $  \norm{\vecx_1 - \vecx_2} \leq \frac{\varepsilon}{\mu}$.
\end{mylemma}
\begin{proof}
  Taylor expanding $\nabla f_1$, we have:
  \begin{align*}
    \nabla f_1(\vecx_2) = \nabla f_1(\vecx_1) + \nabla^2 f_1(\tilde{\vecx}) (\vecx_2 - \vecx_1) = \nabla^2 f_1(\tilde{\vecx}) (\vecx_2 - \vecx_1) \:.
  \end{align*}
  for $\tilde{\vecx} = t \vecx_1 + (1-t) \vecx_2$ with some $t \in [0, 1]$.
  Therefore:
  \begin{align*}
    \mu \norm{\vecx_1 - \vecx_2} \leq \norm{\nabla^2 f_1(\tilde{\vecx}) (\vecx_2 - \vecx_1)} = \norm{\nabla f_1(\vecx_2)} \leq \varepsilon \:.
  \end{align*}
\end{proof}

\begin{mylemma}
  \label{lem:K_perturb_bound}
  Define $f_i(\vecu; \vecx) = \frac{1}{2} \vecu^\T R \vecu + \frac{1}{2}(A_i \vecx + B_i \vecu)^\T
  P_i (A_i \vecx + B_i \vecu)$ for $i=1, 2$, with $R$, $P_1$, and $P_2$  positive definite matrices.  Let $K_i$ be the unique matrix such that
  $\vecu_i := \arg\min_{\vecu} f_i(\vecu; \vecx) =  K_i \vecx$ for any vector $\vecx$. Also, denote $\Gamma := 1 + \max\{\norm{A_1},\norm{B_1},\norm{P_1}, \norm{K_1} \}$. Suppose there exists $\varepsilon$ such that $0 \leq \varepsilon < 1$ and $\norm{A_1-A_2} \leq \varepsilon$, $\norm{B_1-B_2} \leq \varepsilon$, and $\norm{P_1-P_2}  \leq \varepsilon$. Then, we have
  \begin{align*}
    \norm{K_1 - K_2} \leq \frac{7 \varepsilon \Gamma^3 }{\smin(R)} \:.
  \end{align*}
\end{mylemma}
\begin{proof}
  We first compute the gradient $\nabla f_i(\vecu; \vecx)$ with respect to $\vecu$:
  \begin{align*}
    \nabla f_i(\vecu; \vecx) = (B_i^\T P_i B_i + R) \vecu + B_i^\T P_i A_i \vecx \:.
  \end{align*}
  Now, we observe that:
  \begin{align*}
    \norm{B_1^\T P_1 B_1 - B_2^\T P_2 B_2} &\leq 7 \Gamma^2 \varepsilon \; \text{ and } \; \norm{B_1^\T P_1 A_1 - B_2^\T P_2 A_2} = 7\Gamma^2 \varepsilon.
  \end{align*}
  Hence, for any vector $\vecx$ with $\norm{\vecx} \leq 1$, we have
  \begin{align*}
    \norm{ \nabla f_1(\vecu; \vecx) - \nabla f_2(\vecu; \vecx) } &\leq 7 \Gamma^2 \varepsilon (\norm{\vecu} + 1).
  \end{align*}
  We can bound $\norm{\vecu_1} \leq \norm{K_1}\norm{\vecx} \leq \norm{K_1}$. Then, from Lemma~\ref{lem:strongly_convex_min_perturb} we obtain
  \begin{align*}
    \smin(R)\norm{(K_1-K_2)\vecx} = \smin(R)\norm{\vecu_1 - \vecu_2} &\leq 7\Gamma^3 \varepsilon.
  \end{align*}
\end{proof}

Recall that  $\Gamma_\star := 1 + \max\{ \norm{\Astar}, \norm{\Bstar}, \norm{\Pstar}, \norm{\Kstar}\}$. Now, we upper bound $\norm{\Khat - \Kstar}$.

\begin{myprop}
  \label{prop:stability_perturb}
  Let $\varepsilon > 0$ such
  that $\norm{\Ahat - \Astar} \leq \varepsilon$ and $\norm{\Bhat - \Bstar} \leq \varepsilon$. Also, let $\norm{\Phat - \Pstar} \leq f(\varepsilon)$ for some function $f$ such that $f(\varepsilon) \geq \varepsilon$. Then, under Assumption~\ref{ass:strg_cvx} we have
  \begin{align}
    \label{eq:controller_perturb}
    \norm{\Khat - \Kstar} \leq 7 \Gamma_\star^3 \, f(\varepsilon) .
  \end{align}
Let $\gamma$ be a real number such that $\rho(\Lstar) < \gamma < 1$. Then, if $f(\varepsilon)$ is small enough so that the right hand side of \eqref{eq:controller_perturb} is smaller than $\frac{1 - \gamma}{2 \transient{\Lstar}{ \gamma}}$, we have
  \begin{align*}
    \tau \left(\Astar + \Bstar K, \frac{1 + \gamma}{2} \right) \leq \transient{\Lstar}{ \gamma}.
  \end{align*}
\end{myprop}
\begin{proof}
 By our assumptions  $\norm{\Ahat - \Astar}$, $\norm{\Bhat - \Bstar}$, and $\norm{\Phat - \Pstar}$ are smaller than $f(\varepsilon)$, and $\smin(R) \geq 1$.  Then, Lemma~\ref{lem:K_perturb_bound} ensures that
  \begin{align*}
    \norm{\Khat - K_\star} \leq 7  \Gamma_\star^3  \, f(\varepsilon).
  \end{align*}

Finally, when $\varepsilon$ is small enough so that the right hand side of \eqref{eq:controller_perturb} is smaller or equal than $\frac{1 - \gamma}{2 \transient{\Astar + \Bstar \Kstar}{ \gamma}}$,  we can apply  Lemma~\ref{lem:lipschitz_matrix_powers}, presented in Section~\ref{sec:riccati}, to guarantee that  $\norm{(\Astar + \Bstar \Khat)^k} \leq \transient{\Astar + \Bstar \Kstar}{ \gamma} \left(\frac{1 + \gamma}{2}\right)^k$ for all $k \geq 0$.
\end{proof}

In order to finish the proof of Theorem~\ref{thm:meta} we need to quantify the suboptimality gap $\Jhat - J_\star$ in terms of the controller mismatch $\Khat - \Kstar$.
For a stable matrix $L$ and a symmetric matrix $M$, we let $\dlyap(L, M)$ denote the solution $X$ to the Lyapunov equation $L^\top X L - X + M = 0$. The following lemma offers a useful second order expansion of the average LQR cost.

\begin{mylemma}[Lemma 12 of \citet{Fazel18}]
\label{lem:second_order_lqr_perturbation}
Let $K$ be an arbitrary static linear controller that stabilizes $(\Astar, \Bstar)$. Denote $\Sigma(K) := \dlyap((\Astar + \Bstar K)^\T, \sigma_w^2 I_n)$ the covariance matrix of the stationary distribution of the closed loop system $\Astar + \Bstar K$.
We have that:
\begin{align}
\label{eq:second_order_expansion}
    J(\Astar, \Bstar, K) - J_\star = \Tr(\Sigma(K) (K - K_\star)^\T (R + \Bstar^\T P_\star\Bstar) (K - K_\star)) \:.
\end{align}
\end{mylemma}
Now, we have the necessary ingredients to complete the proof of Theorem~\ref{thm:meta}.
  Equation~\ref{eq:second_order_expansion} implies:
\begin{align*}
J(\Astar, \Bstar, K) - J_\star &\leq \norm{\Sigma(K)} \norm{R + \Bstar^\T \Pstar \Bstar} \norm{K - K_\star}_F^2 \:.
\end{align*}
Proposition~\ref{prop:stability_perturb} states that $\Khat$ stabilizes the system $(\Astar, \Bstar)$ when the estimation error is small enough. More precisely, under the assumptions of Theorem~\ref{thm:meta}, we have
$\tau \left(\Astar + \Bstar \Khat, \frac{1 + \gamma}{2} \right) \leq \transient{\Lstar}{ \gamma}$.
When $\widehat{L} = \Astar + \Bstar \Khat$ is a stable matrix we know that $\Sigma(K) = \sigma^2 \sum_{t \geq 0} (L^\top)^t L^t $. Then, by the triangle inequality we can bound 
\begin{align*}
    \norm{\Sigma(K)} \leq \frac{  \sigma_w^2 \transient{\Lstar}{\gamma}^2}{1-\left(\frac{\gamma + 1}{2}\right)^2} \leq \frac{4  \sigma_w^2 \transient{\Lstar}{\gamma}^2}{1 - \gamma^2}\:.
\end{align*}
Recalling that $\Gamma_\star := 1 + \max\{ \norm{\Astar}, \norm{\Bstar}, \norm{\Pstar}, \norm{\Kstar}\}$, we have $\norm{R + \Bstar^\T \Pstar \Bstar} \leq \Gamma^3$. Then,
\begin{align*}
  J(K) - J_\star &\leq 4 \sigma_w^2 \Gamma^3  \frac{ \transient{\Lstar}{\gamma}^2}{1- \gamma^2} \norm{K - K_\star}_F^2\\
                 &\leq 4 \sigma_w^2 \min \{ \statedim , \inputdim \} \Gamma^3 \frac{ \tau(\Lstar, \gamma)^2}{1- \gamma^2}  \norm{K - K_\star}^2 \\
  &\leq 200 \sigma_w^2 d \Gamma^9 \frac{ \tau(\Lstar, \gamma)^2}{1- \gamma^2} f(\varepsilon)^2,
\end{align*}
where we used Proposition~\ref{prop:stability_perturb} and the assumption on $f(\varepsilon)$.

%% file: lqg.tex
%!TEX root = paper.tex

\section{Main Results for the Linear Quadratic Gaussian Problem}
\label{sec:lqg}

Now we consider partially observable systems. In this case the system dynamics have the form:
\begin{subequations}
\begin{align}
	\vecx_{t+1} &= \Astar \vecx_t + \Bstar \vecu_t + \vecw_t \:, \:\: \vecw_t \sim \calN(0, \sigma_w^2 I) \:, \label{eq:state_update} \\
	\vecy_t &= \Cstar \vecx_t + \vecv_t \:, \:\: \vecv_t \sim \calN(0, \sigma_v^2 I)  \label{eq:state_obs} \:.
\end{align}\label{eq:lqg_equations}%
\end{subequations}
In \eqref{eq:lqg_equations}, 
only the output process  
$\vecy_t$ is observed.
The LQG problem is defined as\footnote{Note that many texts define the LQG cost in terms of $\vecx_t^\T Q \vecx_t$ instead of $\vecy_t^\T Q \vecy_t$. We choose the latter because
we do not want the cost to be tied to a particular (unknown) state representation.}:
\begin{align}
  \label{eq:lqg}
  \min_{\vecu_0, \vecu_1, \ldots} &\lim_{T \to \infty} \EE \left[ \frac{1}{T} \sum_{t = 0}^T \vecy_t^\top Q \vecy_t + \vecu_t^\top R \vecu_t \right] ~~\text{s.t.}~~ \eqref{eq:state_update},\eqref{eq:state_obs} \:.
\end{align}
Here, the input $\vecu_t$ is allowed to depend on the history\footnote{The one step delay in $\vecy_t$ is a standard assumption in controls which
slightly simplifies the Kalman filtering expressions. Our results generalize to the setting
where the history also contains the current observation $\vecy_t$.} $\calH_t := (\vecu_0, ..., \vecu_{t-1}, \vecy_0, ..., \vecy_{t-1})$.
The optimal solution to \eqref{eq:lqg} is to set $\vecu_t = K_\star \hatvecx_t$,
with $K_\star$ the optimal LQR solution to $(\Astar, \Bstar, \Cstar^\T Q \Cstar, R)$ and
$\hatvecx_t := \E[ \vecx_t | \calH_t ]$. 
The MSE estimate $\hatvecx_t$ can be solved efficiently via Kalman filtering:
\begin{subequations}
\begin{align}
	\hatvecx_{t+1} &= \Astar \hatvecx_t + \Bstar \vecu_t + L_\star (\vecy_t - \Cstar \hatvecx_t) \:, \\
	L_\star &= - \Astar \Sigma_\star \Cstar^\T (\Cstar \Sigma_\star \Cstar^\T + V)^{-1} \:, \\
	\Sigma_\star &= \Astar \Sigma_\star \Astar^\T + \sigma_w^2 I - \Astar \Sigma_\star \Cstar^\T ( \Cstar \Sigma_\star \Cstar^\T + \sigma_v^2 I)^{-1} \Cstar \Sigma_\star \Astar^\T \:. 
\end{align}
\end{subequations}
There is an inherent ambiguity in  the dynamics \eqref{eq:state_update}-\eqref{eq:state_obs} which makes LQG more delicate than LQR. In particular, for any invertible $T$, the LQG problem \eqref{eq:lqg} with parameters $(\Astar, \Bstar, \Cstar, Q, R)$ is equivalent to the LQG problem with parameters $(T \Astar T^{-1}, T \Bstar, \Cstar T^{-1}, Q, R)$ and appropriately rescaled noise processes. 
To deal with this ambiguity, we assume that we have estimates $(\Ah, \Bh, \Ch, \Lh)$ such that there exists
an unitary $T$ such that:
% \begin{align}
%     \norm{\Ah - T \Astar T^{-1}} \leq \varepsilon_A \:,
%     \norm{\Bh - T \Bstar} \leq \varepsilon_B \:,
%     \norm{\Ch - \Cstar T^{-1}} \leq \varepsilon_C \:,
%     \norm{\Lh - T L_\star} \leq \varepsilon_L \:.
% \end{align} 
\begin{align}
    \max\{ \norm{\Ah - T \Astar T^{-1}}, 
           \norm{\Bh - T \Bstar},
           \norm{\Ch - \Cstar T^{-1}},
           \norm{\Lh - T L_\star} \} \leq \varepsilon \:.
\end{align} 
Recent work \cite{oymak2018non,simchowitz2019learning,tsiamis19} has shown how to obtain this style
of estimates with guarantees from input/output data. 
As in Section~\ref{sec:main}, we assume that the cost matrices $(Q, R)$ are known.
Then, we study the performance of the certainty equivalence controller defined by:
\begin{align}
  \hatvecx_{t+1} &= \Ah \hatvecx_t + \Bh \vecu_t + \Lh (\vecy_t - \Ch \hatvecx_t) \:, \;\;\vecu_t = \Kh \hatvecx_t \:, \;\; \Kh = \mathsf{LQR}(\Ah, \Bh, \Ch^\T Q \Ch, R) \:.
 \label{eq:lqg_CE} 
\end{align}
Similarly to Theorem~\ref{thm:meta} for LQR, we state a meta theorem for LQG. 
Unlike Theorem~\ref{thm:meta}, however, we need a stronger type of Riccati perturbation guarantee
which also allows for perturbation of the $Q$ matrix.
Specifically, we suppose there exists $\gamma_0$ such that for any $\gamma \leq \gamma_0$ and $(\Ah, \Bh, \Qh)$ with $\max\{\norm{\Ah-A},\norm{\Bh-B},\norm{\Qh-Q}\} \leq \gamma$ the solutions $P$ and $\Ph$ of the Riccati equations with parameters $(A,B,Q,R)$ and $(\Ah, \Bh, \Qh, R)$ satisfy
\begin{align}
	\norm{P - \Ph} \leq f(\gamma) \:, \label{eq:dare_perturbation_assumption}
\end{align}
for an increasing function $f$ with $f(\gamma) \geq \gamma$. 
The constant $\gamma_0$ and function $f$ are allowed to depend on the parameters $(A, B, Q, R)$.
In Section~\ref{sec:riccati}, we present a perturbation bound (Proposition~\ref{prop:93})
that satisfies these properties.
Similarly to Section~\ref{sec:main}, let $\Gamma_\star := 1 + \max\{ \norm{\Astar}, \norm{\Bstar}, \norm{C_\star}, \norm{K_\star}, \norm{L_\star}, \norm{P_\star} \}$. 
The following theorem is our main result for LQG.
\begin{mythm}
\label{thm:lqg_subopt}
Suppose that $(\Astar, \Bstar)$ is stabilizable, $(\Cstar, \Astar)$ is observable,
and that Assumption~\ref{ass:strg_cvx} holds.
Let $\varepsilon$ be an upper bound on
$\norm{\Ah - T A_\star T^{-1}}$, $\norm{\Bh - T B_\star}$, $\norm{\Ch - C_\star T^{-1}}$, and $\norm{\Lh- T L_\star}$
for some unitary transformation $T$. Suppose that assumption \eqref{eq:dare_perturbation_assumption} holds 
with parameters $T \Astar T^{-1}$, $T \Bstar$, $T^{-\T} \Cstar^\T Q \Cstar T^{-1}$, and $R$
and that $\varepsilon$ is sufficiently small so that $ 3 \norm{\Cstar}_+\norm{Q}_+ \varepsilon \leq \gamma_0$ and $\overline{\varepsilon} \leq 1$, where
$\overline{\varepsilon} := \frac{7 \Gamma_\star^3}{\smin(R)} f( 3 \norm{\Cstar}^2_+\norm{Q}_+ \varepsilon)$.
Let $\Kh$ be defined as in \eqref{eq:lqg_CE}, and define $N_\star$ as
\begin{align}
    N_\star := \bmattwo{\Astar+\Bstar\Kstar}{\Bstar\Kstar}{0}{\Astar-\Lstar\Cstar} \:, \label{eq:lqg_N}
\end{align}
where the pair $(K_\star,L_\star)$ is optimal for the LQG problem defined by $(\Astar,\Bstar,\Cstar,Q,R)$.
Let $\gamma > 0$ be such that $\rho(N_\star) < \gamma < 1$. Then as long as
$\overline{\varepsilon} \leq \frac{1-\gamma}{20 \Gamma_\star \tau(N_\star, \gamma)}$,
the interconnection of \eqref{eq:lqg_CE} with \eqref{eq:lqg_equations} using $(\Ah, \Bh, \Ch, \Kh, \Lh)$ is stable. Furthermore,
the cost $J(\Ah,\Bh,\Ch,\Kh,\Lh)$ satisfies:
\begin{align*}
    J(\Ah,\Bh,\Ch,\Kh,\Lh) - J_\star \leq \Ocal(1) \max\{\sigma_w^2, \sigma_v^2\}  (\Tr(\Cstar^\T Q \Cstar)+\Tr(R)) \frac{\tau^6(N_\star,\gamma)}{(1-\gamma^2)^3}\Gamma_\star^6 \overline{\varepsilon}^2 \:.
\end{align*}
\end{mythm}
The proof of Theorem~\ref{thm:lqg_subopt} appears in Appendix~\ref{sec:proofs:lqg}.
We note that such a $\gamma$ exists since $\rho(N_\star) < 1$;
by the stability and observability assumptions in Theorem~\ref{thm:lqg_subopt},
we have that both $\Astar+\Bstar\Kstar$ and $\Astar-\Lstar\Cstar$ are stable
(c.f.~Appendix E of~\citet{kailath00}).
Combining Theorem~\ref{thm:lqg_subopt} with Proposition~\ref{prop:93}, we have the following
analogue of Theorem~\ref{thm:fast_rate} for LQG.
\begin{mythm}
\label{thm:fast_rate_lqg}
Suppose that $(\Astar, \Bstar)$ is stabilizable, $(\Cstar, \Astar)$ is observable,
and that Assumption~\ref{ass:strg_cvx} holds.
Let $\varepsilon$ be an upper bound on
$\norm{\Ah - T A_\star T^{-1}}$, $\norm{\Bh - T B_\star}$, $\norm{\Ch - C_\star T^{-1}}$, and $\norm{\Lh- T L_\star}$
for some unitary transformation $T$.
Let $P_\star = \mathsf{dare}(\Astar,\Bstar,\Cstar^\T Q \Cstar, R)$ and suppose
that $\smin(\Pstar) \geq 1$.
Let $N_\star$ be as in \eqref{eq:lqg_N} and
fix $\gamma$ such that $\rho(N_\star) < \gamma < 1$.
As long as $\varepsilon$ satisfies $\varepsilon \leq \frac{(1-\gamma^2)^2}{\tau^4(N_\star, \gamma)} \frac{1}{\Gamma_\star^{11} \norm{Q}}$, 
we have the following sub-optimality bound:
\begin{align*}
J(\Ah,\Bh,\Ch,\Kh,\Lh) - J_\star \leq \Ocal(1) \max\{\sigma_w^2, \sigma_v^2\} (\Tr(\Cstar^\T Q \Cstar)+\Tr(R)) \frac{\norm{Q}^2}{\smin(R)^2} \Gamma_\star^{26} \frac{\tau^{10}(N_\star, \gamma)}{(1-\gamma^2)^5} \varepsilon^2 \:.
\end{align*}
\end{mythm}

Several remarks are in order. First, the assumption that $\smin(P_\star) \geq 1$ is 
without loss of generality, since we can always rescale $Q$ and $R$ without affecting the control solution.
Next, we compare our results here to a classic result from \citet{doyle78}, which states that
there are no gain margins for LQG. We remark that the notion of a gain margin is 
a robustness property that holds uniformly over a class of perturbations of varying degree.
Our results do not hold uniformly; we use quantities such as $\tau(N_\star, \gamma)$ and $\Gamma_\star$
to quantify how much mismatch a given LQG instance can tolerate.

%% file: riccati_perturbation_main.tex
%!TEX root = paper.tex

\section{Riccati Perturbation Theory}
\label{sec:riccati}

%\stephen{TODO: this section will need to be re-written to account for $Q$ perturbation.}

% In light of  Theorem~\ref{thm:meta}, to show that $\Jhat - \Jstar = \calO(\varepsilon^2)$, it suffices to prove that the discrete Riccati equation is locally Lipschitz with respect to the problem parameters:
% $\norm{\Phat - \Pstar} \leq L\varepsilon$ if $\varepsilon < b$, for some $b$ and $L$. 

As discussed in Sections~\ref{sec:main} and \ref{sec:lqg}, a key piece of our analysis
is bounding the solutions to discrete Riccati equations as we perturb the problem parameters.
Specifically, we are interested in 
quantities $b, L$ such that $\norm{\Phat - \Pstar} \leq L\varepsilon$ if $\varepsilon < b$,
where $\varepsilon$ represents a bound on the perturbation. 
We note that it is not possible to find universal values $b, L$.
Consider the systems $(\Astar, \Bstar) = (1, \varepsilon)$ and
$(\Ah, \Bh) = (1, 0)$; the latter system is not stabilizable and hence $\Phat$ does not even exist.
Therefore, $b$ and $L$ must depend on the system parameters.

% %
% However, we note that one cannot hope to find universal values $b$ and $L$ such that for any $0 < \varepsilon < b$ one has  $\| \Phat - \Pstar \| \leq L \varepsilon$ for arbitrary $(\Astar, \Bstar)$ and $(\Ahat, \Bhat)$ with $\|\Ahat - \Astar \| \leq \varepsilon$ and $\|\Bhat - \Bstar \| \leq \varepsilon$. To see this, consider the one dimensional linear system ($\statedim=1$) given by $\Astar = 1$ and $\Bstar = \varepsilon$ and consider the estimated system $\Ahat = 1$ and $\Bhat = 0$. Then, the estimated system is $\varepsilon$ close to the optimal system, but the estimated system is not stabilizable and hence $\Phat$ is not finite. Even when $\Bhat = \varepsilon / 2$, there is no universal $L$ such that the desired inequality holds for all positive $\varepsilon$. Therefore, $b$ and $L$ must depend on the system parameters.

While there is a long line of work analyzing perturbations of Riccati equations, we are not aware of any result that offers explicit and easily interpretable $b$ and $L$ for a fixed $(\Astar, \Bstar, Q, R)$; see \citet{konstantinov2003perturbation} for an overview of this literature.
In this section, we present two new results for Riccati perturbation which offer interpretable bounds.
The first one expands upon the operator-theoretic proof of \citet{konstantinov1993perturbation}; its proof can be found in Section~\ref{sec:proof_93}.
In this result we assume the cost matrix $Q$ can also be perturbed, which is needed for our LQG guarantee. In order to be consistent we denote the true cost matrix by $Q_\star$ and the estimated one by $\Qh$. 
%In Appendix~\ref{app:93} we expand upon the arguments of \citet{konstantinov1993perturbation} to prove the following explicit perturbation result.

\begin{myprop}
  \label{prop:93}
  Let $\gamma \geq \rho(\Lstar)$ and also let $\varepsilon$ such that $\norm{\Ahat - \Astar}$, $\norm{\Bhat - \Bstar}$, and $\norm{\Qh - Q_\star}$ are at most $\varepsilon$. Let $\norm{\cdot}_+ = \norm{\cdot} + 1$. We assume that $R \succ 0$, $(\Astar, \Bstar)$ is stabilizable, $(Q^{1/2}, \Astar)$ observable, and $\smin(\Pstar) \geq 1$.  
 \begin{align*}
\norm{\Phat - \Pstar} \leq \Ocal(1)\, \varepsilon \, \frac{\transient{\Lstar}{\gamma}^2}{1 - \gamma^2} \norm{\Astar}_+^2 \norm{\Pstar}_+^2 \norm{\Bstar}_+ \norm{R^{-1}}_+,
 \end{align*}
as long as
  \begin{align*}
  \varepsilon \leq \Ocal(1) \frac{(1 - \gamma^2)^2}{\transient{\Lstar}{\gamma}^4} \norm{\Astar}_+^{-2} \norm{\Pstar}_+^{-2} \norm{\Bstar}_+^{-3} \norm{R^{-1}}_+^{-2} \min \left\{ \norm{\Lstar}_+^{-2}, \norm{\Pstar}_+^{-1} \right\}.
  \end{align*}
\end{myprop}
We note that the assumption $\smin(\Pstar) \geq 1$ can be made without loss of generality when the other assumptions are satisfied. When $R \succ 0$ and $(Q^{1/2}, A)$ observable, the value function matrix $\Pstar$ is guaranteed to be positive definite. Then, by rescalling $Q$ and $R$ we can ensure that $\smin(\Pstar) \geq 1$.

We now present our direct approach, which uses Assumption~\ref{ass:controllable} to give a bound which is sharper for some systems $(\Astar, \Bstar)$ then the one provided by Proposition~\ref{prop:93}. Recall that any controllable system is always $(\ell, \lbsmin)$-controllable for some $\ell$ and $\lbsmin$.
\begin{myprop}
\label{prop:lipschitz_dare}
Let $\rho \geq \rho(\Astar)$ and also let $\varepsilon \geq 0$ such that  $\|\Ahat - \Astar\| \leq \varepsilon$ and $\|\Bhat - \Bstar\| \leq \varepsilon$. Let $\beta := \max\{1, \varepsilon \transient{\Astar}{\rho} + \rho\}$. Under Assumptions~\ref{ass:strg_cvx} and \ref{ass:controllable} we have
\begin{align*}
  \|\widehat{P} - \Pstar \|  \leq 32 \, \varepsilon \, \ell^{\frac{5}{2}} \transient{\Astar}{\rho}^3  \beta^{2(\ell - 1)}  \left( 1 + \frac{1}{\lbsmin} \right) ( 1 + \|\Bstar \| )^{2} \|\Pstar \|  \frac{\max \{\|Q\| ,  \|R\| \} }{\min \{ \smin(R) , \underline{\sigma}(Q) \} },
\end{align*}
as long as $\varepsilon$ is small enough so that the right hand side is smaller or equal than one.
\end{myprop}

The proof of this result is deferred to Section~\ref{sec:proof_dare}.
We note that Proposition~\ref{prop:lipschitz_dare} can also be extended to handle perturbations in the cost matrix $Q$, as we describe in the proof. 
Proposition~\ref{prop:lipschitz_dare} requires an $(\ell,\lbsmin)$-controllable
system $(\Astar, \Bstar)$, whereas Proposition~\ref{prop:93} only requires
a stabilizable system, which is a milder assumption. However,
Proposition~\ref{prop:lipschitz_dare} can offer a sharper guarantee. For
example, consider the linear system with two dimensional states ($\statedim=2$)
given by $\Astar = 1.01 \cdot I_2$ and $\Bstar = \begin{bmatrix} 1 & 0\\ 0 &
\beta \end{bmatrix}$. Both $Q$ and $R$ are chosen to be the
identity matrix $I_2$.
This system $(\Astar, \Bstar)$ is readily checked to be $(1, \beta)$-controllable.
%
%TThe only quantities which depend on $\beta$ in the right hand side of  Proposition \ref{prop:lipschitz_dare} are $\nu$ and $\norm{\Pstar}$. The only quantities which depend on $\beta$ in the right hand side of  Proposition \ref{prop:93} are $\gamma$ and $\norm{\Pstar}$. In Figure \ref{fig:} we plot the dependence on $\beta$ of both guarantees and we see that Proposition~\ref{prop:lipschitz_dare} offers a sharper upper bound.
It is also straightforward to verify that as $\beta$ tends to zero,
Proposition~\ref{prop:93}
gives a bound of $\norm{\Phat-\Pstar} = \Ocal(\varepsilon / \beta^4)$,
whereas Proposition~\ref{prop:lipschitz_dare}
gives a sharper bound of $\norm{\Phat-\Pstar} = \Ocal(\varepsilon / \beta^3)$.

\subsection{Proof of Proposition~\ref{prop:93}}
\label{sec:proof_93}

Given parameters $(A,B, Q)$ ($R$ is assumed fixed throughout; $Q$ is assumed positive semidefinite throughout) we denote by $F(X, A, B, Q)$ the matrix expression
\begin{align}
  F(X, A, B, Q) &= X - A^\top X A + A^\top X B (R + B^\top X B)^{-1} B^\top X A - Q \nonumber \\
             &= X - A^\top X\left( I + B R^{-1} B^\top X  \right)^{-1}A - Q \:. \label{eq:riccati_op_short}
\end{align}
Then, solving the Riccati equation associated with $(A, B, Q)$ corresponds to finding the unique positive definite matrix $X$ such that $F(X, A, B, Q) = 0$. We denote by $\Pstar$ the solution of the Riccati equation corresponding to the true system parameters $(\Astar, \Bstar)$ and we denote by $\Phat$ the solution associated with $(\Ahat, \Bhat, \Qh)$. Our goal is to upper bound $\norm{\Phat - \Pstar}$ in terms of $\varepsilon$, where $\varepsilon > 0$ such that $\norm{\Ahat - \Astar} \leq \varepsilon$, $\norm{\Bhat - \Bstar} \leq \varepsilon$, and  $\norm{\Qh - Q_\star} \leq \varepsilon$.

We denote $\Delta_P = \Phat - \Pstar$. The proof strategy goes as follows. Given the identities $F(\Pstar, \Astar, \Bstar, Q_\star) = 0$ and $F(\Phat, \Ahat, \Bhat, \Qh) = 0$ we construct an operator $\opPhi$ such that $\Delta_P$ is its unique fixed point. Then, we show that the fixed point of $\opPhi$ must have small norm when $\varepsilon$ is sufficiently small.

We denote $S_\star = \Bstar R^{-1} \Bstar^\top$ and $\Shat = \Bhat R^{-1} \Bhat^\top$. Also, recall that $L_\star = \Astar + \Bstar \Kstar$. For any matrix $X$ such that $I + \Sstar(\Pstar + X)$ is invertible we have
\begin{align}
  \label{eq:magical}
F(\Pstar + X, \Astar, \Bstar, Q_\star) &= X - L_\star^\top X L_\star + L_\star^\top X \left[I + S_\star (\Pstar + X)\right]^{-1} S_\star X L_\star.
\end{align}
To check this identity one needs to add  $F(\Pstar, \Astar, \Bstar, Q_\star)$, which is equal to zero, to the right hand side of \eqref{eq:magical} and use the identity $(I + \Bstar R^{-1} \Bstar^\top \Pstar)^{-1} \Astar = \Astar + \Bstar \Kstar$. This last identity can be checked by recalling $\Kstar = - (R + \Bstar^\top \Pstar \Bstar)^{-1} \Bstar^\top \Pstar \Astar$ and using the matrix inversion formulat.

To write \eqref{eq:magical} more compactly we define the following two matrix operators
\begin{align*}
\opT(X) = X - L_\star^\top X L_\star \;\text{ and }\; \opH(X) = L_\star^\top X \left(I + S_\star(\Pstar + X)\right)^{-1} S_\star X L_\star.
\end{align*}
Then, Equation~\ref{eq:magical} becomes $F(\Pstar + X, \Astar, \Bstar, Q_\star) = \opT(X) + \opH(X)$.
Since Equation~\ref{eq:magical} is satisfied by any matrix $X$ with $I + S_\star(\Pstar+X)$ invertible, the matrix equation
\begin{align}
  \label{eq:fixed_point_eq}
  F(\Pstar + X, \Astar, \Bstar, Q_\star) - F(\Pstar + X, \Ahat, \Bhat, \Qh) = \opT(X) + \opH(X)
\end{align}
 has a unique symmetric solution $X$ such that $\Pstar + X \succeq 0$.
That solution is $X = \Delta_P$ because any solution of \eqref{eq:fixed_point_eq} must satisfy $F(\Pstar + X, \Ahat, \Bhat, \Qh) = 0$.

The linear map $\opT \colon X \mapsto X - \Lstar^\top X \Lstar$ has eigenvalues equal to $1 - \lambda_i \lambda_j$, where $\lambda_i$ and $\lambda_j$ are eigenvalues of the closed loop matrix $\Lstar$. Since $\Lstar$ is a stable matrix, the linear map $\opT$ must be invertible. Now, we define the operator
\begin{align*}
  \opPhi(X) &= \opT^{-1} \left(F(\Pstar + X, \Astar, \Bstar, Q_\star) - F(\Pstar + X, \Ahat, \Bhat, \Qh) - \opH(X)\right).
\end{align*}
Then, solving for $X$ in Equation~\ref{eq:fixed_point_eq} is equivalent to finding $X$ satisfying $\Pstar + X \succeq 0$ such that $X = \opPhi(X)$. Hence, $\opPhi$ has a unique symmetric fixed point $X$ such that $\Pstar + X \succeq 0$ and that is $X = \Delta_P$. Now, we consider the set
\begin{align*}
\calS_\nu := \left\{X \colon \norm{X} \leq \nu \text{, } X = X^\top\text{, } \Pstar + X \succeq 0 \right\}
\end{align*}
and we show that for an appropriately chosen $\nu$ the operator $\opPhi$ maps $\calS_\nu$ into itself and is also a contraction over the set $\calS_\nu$. If we show these two properties, $\opPhi$ is guaranteed to have a fixed point in the set $\calS_\nu$. However, since $\Delta_P$ is the only possible fixed point of $\opPhi$ in a set $\calS_\nu$ we find $\norm{\Delta_P} \leq \nu$.

We denote $\Delta_A = \Ahat - \Astar$, $\Delta_B = \Bhat - \Bstar$, $\Delta_Q = \Qh - Q_\star$, and $\Delta_S = \Shat - \Sstar$. By assumption we have $\norm{\Delta_A} \leq \varepsilon$,  $\norm{\Delta_B} \leq \varepsilon$, $\norm{\Delta_Q} \leq \varepsilon$. Then, $\norm{\Delta_S} \leq 3 \norm{\Bstar} \norm{R^{-1}}\varepsilon$ because  $\varepsilon \leq \norm{\Bstar}$. %The next lemma assumes $\Qh = Q_\star$; its proof relies on just algebraic computations which can be extended to the case $\Qh \neq Q_\star$. 
\begin{mylemma}
  \label{lem:operator_93}
  Suppose the matrices $X$, $X_1$, $X_2$ belong to $\calS_\nu$, with $\nu \leq \min \{ 1, \norm{\Sstar}^{-1}\}$. Furthermore, we assume that $\norm{\Delta_A} \leq \varepsilon$, $\norm{\Delta_B} \leq \varepsilon$, and $\norm{\Delta_Q} \leq \varepsilon$ with $\varepsilon \leq \min \{1, \norm{\Bstar} \}$. Finally, let $\smin(\Pstar) \geq 1$. Then
  \begin{align*}
    \norm{\opPhi(X)} &\leq 3 \frac{\transient{\Lstar}{\gamma}^2}{1 - \gamma^2} \left[\norm{\Lstar}^2 \norm{\Sstar} \nu^2 +  \varepsilon \norm{\Astar}_+^2 \norm{\Pstar}^2_+ \norm{\Bstar}_+ \norm{R^{-1}}_+  \right],\\
    \norm{\opPhi(X_1) - \opPhi(X_2)} &\leq 32 \frac{\transient{\Lstar}{\gamma}^2}{1 - \gamma^2} \left[ \norm{\Lstar}^2 \norm{\Sstar} \nu + \varepsilon \norm{\Astar}_+^2 \norm{\Pstar}^3_+ \norm{\Bstar}^3_+ \norm{R^{-1}}^2_+  \right] \norm{X_1 - X_2}.
  \end{align*}
\end{mylemma}
The proof of this lemma is defered to Appendix~\ref{app:lem_93}. Now, we choose
\begin{align}
  \label{eq:nu_choice}
  \nu = 6 \, \varepsilon \, \frac{\transient{\Lstar}{ \gamma}^2}{1 - \gamma^2} \norm{\Astar}_+^2 \norm{\Pstar}_+^2\norm{\Bstar}_+ \norm{R^{-1}}_+.
\end{align}
Since $\varepsilon$ is assumed to be small enough, we know
\begin{align*}
\nu  \leq \min \left\{ \frac{1 - \gamma^2}{128 \transient{\Lstar}{\gamma}^2 \norm{\Lstar}^2 \norm{\Sstar}}, \norm{\Sstar}^{-1},  \frac{1}{2} \right\}.
\end{align*}
Then, the operator $\opPhi$ satisfies $\norm{\opPhi(X_1) - \opPhi(X_2)} \leq \frac{1}{2}\norm{X_1 - X_2}$ for all $X_1$ and $X_2$ in $\calS_\nu$. Moreover, we have $\norm{\opPhi(X)} \leq \nu$ for all $X\in \calS_\nu$. Since $\nu \leq \smin(\Pstar)$, we know that $\Pstar + \opPhi(X) \succeq 0$ 

Therefore, $\opPhi$ maps $\calS_\nu$ into itself and is a contraction over $\calS_\nu$. Hence, $\opPhi$ has a fixed point in $\calS_\nu$ since $\calS_\nu$ is a closed set. However, we already argued that the unique fixed point of $\opPhi$ is $\Delta_P$. Therefore, $\Delta_P \in \calS_\nu$  and $\norm{\Delta_P} \leq \nu$. Proposition \ref{prop:93} is now proven.

\subsection{Proof of Proposition~\ref{prop:lipschitz_dare}}
\label{sec:proof_dare}

Since both noisy and noiseless LQR have the same associated Riccati equation and the same optimal controller, we can focus on the noiseless case in this section. Namely, noiseless LQR takes the form
\begin{align*}
  \min_{\vecu} \sum_{t = 0}^\infty \vecx_t^\top Q \vecx_t + \vecu_t^\top R \vecu_t \; \text{, where } \; \vecx_{t + 1} = \Astar \vecx_t + \Bstar \vecu_t,
\end{align*}
for a given initial state $\vecx_0$. Then, we know that the cost achieved by the optimal controller
when the system is initialized at $\vecx_0$ is equal to $\vecx_0^\top \Pstar \vecx_0$.

We denote by
$J(A, B, \vecx_0, \{\vecu_t \}_{t \geq 0})$ the cost achieved on a
linear system $(A, B)$ initialized at $\vecx_0$ by the input sequence $\{ \vecu_t\}_{t \geq 0}$. When the input sequence is
given by a time invariant linear gain matrix $K$ we slightly abuse notation and denote the cost by $J(A, B, \vecx_0, K)$. In this case, $J(A, B, \vecx_0, K)= \vecx_0^\top P \vecx_0$, where $P$ is the solution to the associated Riccati equation.

Now, let $\vecx_0$ be an arbitrary unit state vector in $\R^\statedim$. Then,
\begin{align*}
  \vecx_0^\top \Phat \vecx_0 - \vecx_0^\top \Pstar \vecx_0 &= J(\Ahat, \Bhat, \vecx_0, \Khat) - J(\Astar, \Bstar, \vecx_0, \Kstar) \\
                                                           &\leq J(\Ahat, \Bhat, \vecx_0, \{ \widehat{\vecu}_t\}_{t \geq 0}) - J(\Astar, \Bstar, \vecx_0, \Kstar)
\end{align*}
for any sequence of inputs $\{\widehat{\vecu}_{t} \}_{t \geq 0}$. We denote by $\hat{\vecx}_t$ the states produced by
$\widehat{\vecu}_t$ on the system $(\Ahat, \Bhat)$ and by $\vecx_t$
and $\vecu_t$ the states and actions obtained on the system
$(\Astar, \Bstar)$ when the optimal controller $\vecu_t = \Kstar \vecx_t$ is used. To prove
Proposition~\ref{prop:lipschitz_dare} we choose a sequence of actions
$ \{ \widehat{\vecu}_t\}_{t \geq 0}$ such that $ J(\Ahat, \Bhat, \vecx_0, \{ \widehat{\vecu}_t\}_{t \geq 0}) \approx
J(\Astar, \Bstar, \vecx_0, \Kstar)$.

For any sequence of inputs $\{\widehat{\vecu}_{t}\}_{t \geq 0}$ such that the series defining the cost $J(\Ahat, \Bhat, \vecx_0, \{ \widehat{\vecu}_t\}_{t \geq 0})$ is
absolutely convergent, we can write
\begin{align}
  \label{eq:block_series}
  J(\Ahat, \Bhat, \vecx_0, \Khat) - J(\Astar, \Bstar, \vecx_0, \Kstar) &= \sum_{j = 0}^{\infty} \sum_{i = 0}^{\ell - 1} \left[ \widehat{\vecx}_{\ell j + i}^\top Q \widehat{\vecx}_{\ell j + i} - \vecx_{\ell j + i}^\top Q  \vecx_{\ell j + i} \right]\\
  \nonumber
                                                                       &\quad + \sum_{j = 0}^{\infty} \sum_{i = 0}^{\ell - 1} \left[  \widehat{\vecu}_{\ell j + i}^\top R \widehat{\vecu}_{\ell j + i}  -  \vecu_{\ell j + i}^\top R   \vecu_{\ell j + i} \right].
\end{align}
Then, the key idea is to choose a sequence of inputs
$\{\widehat{\vecu}_{t} \}_{t \geq 0}$ such that the system
$(\Ahat, \Bhat)$ tracks the system $(\Astar, \Bstar, \Kstar)$, i.e.,
$\widehat{\vecx}_{\ell j} = \vecx_{\ell j}$ for any $j \geq 0$ (
$\widehat{\vecx}_0 = \vecx_0$ because both systems are initialized at
the same state). This can be done because $(\Ahat, \Bhat)$ is $(\ell, \tau / 2)$-controllable when $(\Astar, \Bstar)$ is $(\ell, \tau)$-controllable and the estimation error is sufficiently small, as shown in Lemma~\ref{lem:controllable_perturb}.
First, we present a result that quantifies the effect of matrix perturbations on powers of matrices.

\begin{mylemma}
  \label{lem:lipschitz_matrix_powers}
  Let $M$ be an arbitrary matrix in $\R^{\statedim \times \statedim}$
  and let $\rho \geq \rho(M)$. Then, for all
  $k \geq 1$ and real matrices $\Delta$ of appropriate dimensions we have

\begin{align*}
   \norm{(M + \Delta)^k} &\leq \transient{M}{\rho} ( \transient{M}{\rho} \norm{\Delta} + \rho)^{k},\\
  \norm{(M + \Delta)^k - M^k} &\leq k \, \transient{M}{\rho}^2 ( \transient{M}{\rho} \norm{\Delta} + \rho)^{k - 1} \norm{\Delta}.
\end{align*}
Recall that $\transient{M}{\rho}$ is defined in Equation~\ref{eq:transient}.
\end{mylemma}
The proof is deferred to Appendix~\ref{app:proof_lemma_matrix_perturb}.
Lemma~\ref{lem:lipschitz_matrix_powers} quantifies the effect of a perturbation $\Delta$,
applied to a matrix $M$ on the spectral radius of $M + \Delta$. We are interested
in quantifying the sizes of these perturbations for all $k = 1, 2, \ldots, \ell$. Depending on $\norm{\Delta}$, $M$, and $\rho$ the sum $\tau(M, \rho) \norm{\Delta} + \rho$ can either be greater than one or smaller than one. For notational
simplicity, in the rest of the proof we denote $\beta = \max\{1, \varepsilon  \transient{\Astar}{\rho}  + \rho\}$. Then,
we have $\norm{(\Astar + \Delta)^k} \leq \transient{\Astar}{\rho} \beta^{\ell - 1}$ and $\norm{(\Astar + \Delta)^k - \Astar^k} \leq \ell \transient{\Astar}{\rho}^2 \beta^{\ell - 1} \varepsilon$ for all $k \leq \ell - 1$ and all real matrices $\Delta$ with $\norm{\Delta} \leq \varepsilon$.

We denote $C_\ell = \begin{bmatrix}
\Bstar & \Astar \Bstar & \ldots & \Astar^{\ell - 1} \Bstar
\end{bmatrix} $ and $\widehat{\Ccal}_\ell = \begin{bmatrix} \Bhat & \Ahat\Bhat & \ldots &
  \Ahat^{\ell - 1}\Bhat \end{bmatrix}$. Before presenting the next result
we recall that for any block matrix $M$ with blocks $M_{i,j}$ we have
$  \|M\|^2 \leq  \sum_{i,j} \|M_{i,j}\|^2$.
The next lemma gives us control over the smallest positive singular value of the controllability matrix $\widehat{\Ccal}_\ell$ in terms of the corresponding value for $\Ccal_\ell$.

\begin{mylemma}
  \label{lem:controllable_perturb}
  Suppose the linear $(\Astar, \Bstar)$ is $(\ell, \lbsmin)$-controllable and let $\rho$ be a real number such that $\rho \geq \rho(\Astar)$. Then, if
  $\| \Ahat - \Astar \| \leq \varepsilon$ and
  $\| \Bhat - \Bstar \| \leq \varepsilon$,
  we have
  \begin{align*}
    \underline{\sigma}(\widehat{C}_\ell) \geq \tau - 3 \varepsilon \ell^{\frac{3}{2}} \transient{\Astar}{\rho}^2  \max\{1, \transient{\Astar}{\rho} \norm{\Delta} + \rho\}^{\ell - 1} \left( \norm{\Bstar} + 1 \right). \end{align*}
\end{mylemma}
The proof is deferred to Appendix~\ref{app:proof_controllable_perturb}.
Lemma~\ref{lem:controllable_perturb} tells us that
by the assumption made in Proposition~\ref{prop:lipschitz_dare} on $\varepsilon$, we have
$\underline{\sigma}(\widehat{\Ccal}_\ell) \geq
\frac{\tau_\ell}{2}$. Hence, we know that for any $\vecx_0\in \RR^\statedim$ and $\vecu_0, \vecu_1, \ldots, \vecu_{\ell - 1}\in \RR^\inputdim$, there
exist $\widehat{\vecu}_0, \widehat{\vecu}_1, \ldots, \widehat{\vecu}_{\ell -
  1} \in \RR^\inputdim$ such that
\begin{align}
  \label{eq:matching_states}
  \Astar^{\ell} \vecx_0 + \sum_{i = 0}^{\ell - 1} \Astar^i \Bstar\vecu_{\ell - 1 - i} = \Ahat^{\ell} \vecx_0 + \sum_{i = 0}^{\ell - 1}\Ahat^i\Bhat\widehat{\vecu}_{\ell - 1 - i}
\end{align}
because the system $(\Ahat, \Bhat)$ is controllable. This equation implies that $\widehat{\vecx}_\ell = \vecx_\ell$.
% In order to avoid cumbersome notation we think inductively and consider each state $\vecx_{\ell j}$
%as the new initial vector $\vecx_0$ in Equation \eqref{eq:block_series}.

We denote the concatenation of $\vecu_i$, for $i$ from $0$ to
$\ell - 1$ by $\vecu^{(\ell)}$. We define $\widehat{\vecu}^{(\ell)}$ analogously. Therefore, Equation~\eqref{eq:matching_states} can be
rewritten as
\begin{align}
  \label{eq:tracking}
  \left( \Astar^\ell - \widehat{A}^\ell  \right) \vecx_0 + \left(\mathcal{C}_\ell - \widehat{\mathcal{C}}_\ell \right) \vecu^{(\ell)} = \widehat{\mathcal{C}}_\ell (\widehat{\vecu}^{(\ell)} - \vecu^{(\ell)}).
\end{align}
Recall that $\beta = \max\{1, \transient{\Astar}{\rho} \norm{\Delta} + \rho\}$.
Combining Lemma~\ref{lem:lipschitz_matrix_powers} and the upper bound on operator norms of block matrices we find $\| \widehat{\Ccal}_\ell - \Ccal_\ell \| \leq \varepsilon \ell^{\frac{3}{2}} \transient{\Astar}{ \rho}^2 \beta^{\ell - 1} \left( \| \Bstar \| + 1 \right)$.

We are free to choose $\widehat{\vecu}^{(\ell)}$ anyway we wish as long as Equation~\eqref{eq:tracking} is true. Therefore, we can choose $\widehat{\vecu}^{(\ell)}$
such that $\widehat{\vecu}^{(\ell)} - \vecu^{(\ell)}$ is perpendicular to the nullspace of $\widehat{\Ccal}_\ell$. Then,
\begin{align*}
  \frac{\tau_\ell}{2} \|\widehat{\vecu}^{(\ell)} - \vecu^{(\ell)} \| &\leq \| \widehat{\Ccal}_\ell(\widehat{\vecu}^{(\ell)} - \vecu^{(\ell)} ) \| \leq \varepsilon \ell \transient{\Astar}{ \rho}^2 \beta^{\ell - 1} \| \vecx_0\| + \|\widehat{\Ccal}_\ell - \Ccal_\ell\| \| \vecu^{(\ell)}\| \\
 &\leq \varepsilon \ell \transient{\Astar}{ \rho}^2 \beta^{\ell - 1} \| \vecx_0\|  + \varepsilon \ell^{\frac{3}{2}} \transient{\Astar}{ \rho}^2 \beta^{\ell - 1} \left(\|\Bstar \| + 1 \right) \| \vecu^{(\ell)}\|.
\end{align*}
Hence,
\begin{align}
    \nonumber
  \|\widehat{\vecu}^{(\ell)} - \vecu^{(\ell)} \|
  &\leq \frac{2 \varepsilon \ell^{\frac{3}{2}}}{\tau_\ell} \transient{\Astar}{ \rho}^2 \beta^{\ell - 1} (\| \Bstar \| + 1) \left( \| \vecx_0 \| + \| \vecu^{(\ell)} \| \right)\\
  \label{eq:u_bound}
  &=: \eta \left( \| \vecx_0 \| + \| \vecu^{(\ell)} \| \right) \:.
\end{align}
Let us consider the block Toeplitz matrix
\begin{align*}
  \Tcal_\ell = \begin{bmatrix}
    0 & 0 & 0 & \ldots & 0\\
    \Bstar & 0 & 0 & \cdots & 0\\
    \Astar \Bstar & \Bstar & 0 & \cdots & 0 \\
    \vdots & \cdots & \cdots & \cdots & \cdots  \\
    \Astar^{\ell - 2} \Bstar & \Astar^{\ell - 3} \Bstar & \cdots & \Bstar & 0
  \end{bmatrix}.
\end{align*}
From Lemma~\ref{lem:lipschitz_matrix_powers} and the upper bound on operator norms of block matrices we have
$\|\Tcal_\ell - \widehat{\Tcal}_\ell\| \leq \varepsilon \ell^2 \transient{\Astar}{\rho}^2 \beta^{\ell - 2}(\|\Bstar \| + 1)$.
% where $\beta =  \max\{1, \transient{\Astar}{\rho} \norm{\Delta} + \rho\}$.
Let $\vecx^{(\ell)}$ be the concatenation of the vectors $\vecx_0$, $\vecx_1$, \ldots, $\vecx_{\ell - 1}$. Then,
\begin{align*}
  \vecx^{(\ell)} = \Tcal_\ell \vecu^{(\ell)} + \begin{bmatrix}
    I_\statedim \\
    \Astar\\
    \vdots\\
    A^{\ell - 1}
  \end{bmatrix}
  \vecx_0.
\end{align*}
Hence,
\begin{align}
  \nonumber
  \| \vecx^{(\ell)}  -  \widehat{\vecx}^{(\ell)} \| &\leq \| \Tcal_\ell \vecu^{(\ell)}  - \widehat{\Tcal}_\ell \widehat{\vecu}^{(\ell)}  \| + \varepsilon \ell^{\frac{3}{2}} \transient{\Astar}{\rho}^2 \beta^{\ell - 2} \| \vecx_0 \| \\
    \nonumber
                                                    &\leq \| \Tcal_\ell \vecu^{(\ell)}  - \widehat{\Tcal}_\ell \vecu^{(\ell)}  \| + \| \widehat{\Tcal}_\ell \vecu^{(\ell)}  - \widehat{\Tcal}_\ell \widehat{\vecu}^{(\ell)}  \| +  \varepsilon \ell^{\frac{3}{2}}  \transient{\Astar}{ \rho}^2 \beta^{\ell - 2} \| \vecx_0 \| \\
    \nonumber
                                                    &\leq \| \Tcal_\ell  - \widehat{\Tcal}_\ell \| \| \vecu^{(\ell)}  \| + \| \widehat{\Tcal}_\ell\| \| \vecu^{(\ell)}  - \widehat{\vecu}^{(\ell)}  \| +  \varepsilon \ell^{\frac{3}{2}}  \transient{\Astar}{\rho}^2 \beta^{\ell - 2} \| \vecx_0 \|  \\
    \nonumber
                                                    &\leq  \varepsilon \ell^2 \transient{\Astar}{\rho}^2 \beta^{\ell - 2}(\|\Bstar \| + 1) \| \vecu^{(\ell)}  \| +  \varepsilon  \ell^{\frac{3}{2}}  \transient{\Astar}{\rho}^2 \beta^{\ell - 2} \| \vecx_0 \|   \\
    \nonumber
                                                    &\quad + \ell \transient{\Astar}{\rho} \beta^{\ell - 2} (\|\Bstar\| + 1) \| \vecu^{(\ell)}  - \widehat{\vecu}^{(\ell)}  \| \\
    \nonumber
                                                    &\leq 2 \varepsilon \ell^{\frac{5}{2}} \transient{\Astar}{\rho}^3 \beta^{2(\ell - 1)} (1 + \tau^{-1}) (\|\Bstar \| + 1)^2  \left[ \| \vecx_0 \| + \|\vecu^{(\ell)}\| \right] \\
  \label{eq:x_bound}
                                                      &=: \mu \left[  \| \vecu^{(\ell)}  \| +  \| \vecx_0 \| \right].
\end{align}
In Equations \eqref{eq:u_bound} and \eqref{eq:x_bound} we proved that the inputs and states of the system $(\Ahat, \Bhat)$ are close to the inputs and states of the system $(\Astar, \Bstar)$ from time $0$ to $\ell$. Since the inputs to the system $(\Ahat, \Bhat)$ satisfy Equation~\eqref{eq:tracking}, we know that $\widehat{\vecx}_{\ell j} = \vecx_{\ell j}$ for all $j$. We can repeat the same argument as above, with $\vecx_{\ell j}$ taking the place of $\vecx_0$, to show that the inputs and states of the two systems are close to each other from time $\ell j$ to $\ell (j + 1)$. Let us denote by $\vecx_j^{(\ell)}$ the concatenation of the vectors $\vecx_{\ell j}$, $\vecx_{\ell j + 1}$, \ldots, $\vecx_{\ell j + \ell - 1}$ and let $\vecu_j^{(\ell)}$ be defined analogously. Then,
\begin{align}
  \label{eq:u_j_bound}
  \|\widehat{\vecu}_j^{(\ell)} - \vecu_j^{(\ell)} \| &\leq  \eta \left[ \| \vecu_j^{(\ell)} \| +  \| \vecx_{\ell j} \| \right], \; \text{ and }\;
  \| \widehat{\vecx}_j^{(\ell)}  -  \vecx_j^{(\ell)}    \| \leq \mu \left[\| \vecu_j^{(\ell)}  \| +  \| \vecx_{\ell j} \| \right].
\end{align}
Now, we note that
\begin{align*}
  \vecx_0^\top \Phat \vecx_0 - \vecx_0^\top \Pstar \vecx_0  &\leq \sum_{j = 0}^{\infty} \sum_{i = 0}^{\ell - 1} \left[ \widehat{\vecx}_{\ell j + i}^\top Q \widehat{\vecx}_{\ell j + i} - \vecx_{\ell j + i}^\top Q  \vecx_{\ell j + i} \right]\\
                         &\quad + \sum_{j = 0}^{\infty} \sum_{i = 0}^{\ell - 1} \left[  \widehat{\vecu}_{\ell j + i}^\top R \widehat{\vecu}_{\ell j + i}  -  \vecu_{\ell j + i}^\top R   \vecu_{\ell j + i} \right]\\
                         &\leq \sum_{j = 0}^{\infty} 2 \|Q \| \| \vecx_j^{(\ell)}\| \| \vecx_{j}^{(\ell)}  -  \widehat{\vecx}_{j}^{(\ell)} \| + \|Q\| \| \vecx_{j}^{(\ell)}  -  \widehat{\vecx}_{j}^{(\ell)} \|^2  \\
                         &\quad +  \sum_{j = 0}^{\infty} 2 \|R \| \| \vecu_{j}^{(\ell)}\| \| \vecu_{j}^{(\ell)}  -  \widehat{\vecu}_{j}^{(\ell)} \| + \|R\| \| \vecu_{j}^{(\ell)}  -  \widehat{\vecu}_{j}^{(\ell)} \|^2 \:.
\end{align*}
%
%Now, we use the upper bounds \eqref{eq:u_j_bound} to show that the right hand side is small.
%Since $\vecx_0$ is an arbitrary unit vector, this analysis shows that $\lambda_{\max}\left(\Phat - \Pstar\right)$ is small. Swapping
%the roles of $(\Ahat, \Bhat)$ and $(\Astar, \Bstar)$ gives us that $\lambda_{\max}\left(\Pstar - \Phat\right)$ is also small, yielding the desired conclusion. We defer the details of this argument to Appendix~\ref{app:remainder_proof_perturb}.
%
Now, we use the upper bounds from \eqref{eq:u_j_bound}. We always have $\eta \leq \mu$. Since Proposition~\ref{prop:lipschitz_dare} assumes $\varepsilon$ is  small enough, we also have $\mu \leq 1$. Using these upper bounds, we find
\begin{align*}
  \widehat{J} - J_\star  &\leq \mu \sum_{j = 0}^{\infty} 2 \|Q \| \| \vecx_j^{(\ell)}\| \left[\| \vecu_j^{(\ell)}  \| +  \| \vecx_{\ell j} \| \right] + \|Q\| \left[\| \vecu_j^{(\ell)}  \| +  \| \vecx_{\ell j} \| \right]^2  \\
                         &\quad + \mu \sum_{j = 0}^{\infty} 2 \|R \| \| \vecu_{j}^{(\ell)}\| \left[\| \vecu_j^{(\ell)}  \| +  \| \vecx_{\ell j} \| \right]  + \|R\| \left[\| \vecu_j^{(\ell)}  \| +  \| \vecx_{\ell j} \| \right]^2 \:.
\end{align*}
Then, we get $\widehat{J} - J_\star \leq 8 \mu \max\{ \|Q\| , \|R\| \} \sum_{j = 0}^\infty \| \vecx_{j}^{(\ell)}\|^2 + \| \vecu_{j}^{(\ell)}\|^2$ after using the inequalities $(a + b)^2 \leq 2(a^2 + b^2)$ and $2ab \leq a^2 + b^2$.
Now, As long as $\norm{\vecx_0} \leq 1$ we have
\begin{align}
  \label{eq:upper_bound_states_series}
  \min \{ \underline{\sigma}(Q) , \underline{\sigma}(R) \} \sum_{j = 0}^\infty  \| \vecx_j^{(\ell)}\|^2 +  \| \vecu_j^{(\ell)}\|^2 \leq \sum_{t = 0}^\infty \vecx_t^\top Q \vecx_t + \vecu_t^\top R \vecu_t =
  \vecx_0^\top \Pstar \vecx_0  \leq \|\Pstar\|.
\end{align}
Since the initial state is an arbitrary unit norm vector, our upper bound on $\vecx_0^\top (\Phat - \Pstar) \vecx_0$ becomes
\begin{align}
  \label{eq:j_jhat_2}
  \lambda_{\max} \left(\widehat{P} - \Pstar \right) \leq 16 \varepsilon \ell^{\frac{5}{2}}  \transient{\Astar}{\rho}^3 \beta^{2(\ell - 1)} (1 + \lbsmin^{-1}) ( 1 + \|\Bstar \| )^2 \|P_\star \| \frac{ \max\{ \|Q\|,   \|R\| \} }{ \min \{ \underline{\sigma}(Q) ,  \underline{\sigma}(R) \} }.
\end{align}
Now, we can reverse the roles of $(\Ahat, \Bhat)$ and
$(\Astar, \Bstar)$ and repeat the same argument and obtain an upper bound
on $\lambda_{\max}\left(\Pstar - \widehat{P}\right)$ analogous to Equation~\eqref{eq:j_jhat_2},
but which has $\|P_\star\|$ replaced by $\|\widehat{P}\|$ on the right hand side. However,
\eqref{eq:j_jhat_2} implies that $\|\widehat{P}\| \leq \|P_\star \| + 1 \leq 2 \|P_\star \|$ because we assumed that $\varepsilon$ is small enough such that the right hand side of \eqref{eq:j_jhat_2} is less than one, and because $\Pstar \succeq I_\statedim$. The conclusion follows.

We note that the proof can be  extended to the case when the cost matrix $Q$ is also being perturbed. Moreover, the only essiential step in the argument where we used $Q \succ 0$ is \eqref{eq:upper_bound_states_series}. The goal of \eqref{eq:upper_bound_states_series} is to upper bound  $\sum_{j = 0}^\infty \| \vecx_{j}^{(\ell)}\|^2 + \| \vecu_{j}^{(\ell)}\|^2$. This quantity can be upper bounded even when $Q$ is not positive definite, but the system $(Q^{1/2}, A)$ is observable; which is a necessary requirement for LQR on the parameters $(A,B,Q,R)$ to stabilize the system $(A,B)$. 

%% file: related.tex
%!TEX root = paper.tex

\section{Related Work}
\label{sec:related}

%Because LQR is a fundamental problem in optimal control, there
%is a vast literature surrounding it. 
%We focus on related work concerning LQ control with unknown dynamics. In particular, we divide existing literature into two main categories: offline batch evaluation and the online adaptive setting.

For the offline LQR batch setting, \citet{fiechter1997pac} proved that the
sub-optimality gap  $\Jhat - \Jstar$ scales as $\calO(\varepsilon)$ for
certainty equivalent control.
A crucial assumption of his analysis is that the nominal  controller stabilizes the true unknown system. We give bounds on when this assumption is valid.
Recently, \citet{dean2017sample} proposed a
robust controller synthesis procedure which takes model uncertainty into account and whose 
 suboptimality gap scales as $\calO(\varepsilon)$.
\citet{tu18} show that the gap $\Jhat - \Jstar$ of certainty equivalent control
scales asymptotically as $\calO(\varepsilon^2)$;
we provide a non-asymptotic analogue of this result. 
\citet{Fazel18} and \citet{malik2018derivative} analyze a model-free approach to policy optimization for LQR, in which the
controller is directly optimized from sampled rollouts. \citet{malik2018derivative}
showed that, after collecting $N$ rollouts, a derivative free method achieves a discounted cost gap that scales as $\calO(1/\sqrt{N})$ or $\calO(1 / N)$, depending on the oracle model used.

In the online LQR adaptive setting it is well understood that using the 
certainty equivalence principle without adequate exploration can result in
a lack of parameter convergence \cite[see e.g.][]{astrom73}. \citet{abbasi2011regret} showed that
optimism in the face of uncertainty (OFU), when applied to online LQR, yields $\Otilde(\sqrt{T})$ regret.
\citet{faradonbeh17} removed some un-necessary assumptions of the previous analysis.
\citet{ibrahimi12} showed that when the underlying system is sparse, the dimension dependent constants
in the regret bound can be improved.
The main issue with OFU for LQR is that there are no known computationally tractable ways of implementing it. 
In order to deal with this, both \citet{dean18regret} and \citet{abbasi18} propose polynomial time algorithms
for adaptive LQR based on $\varepsilon$-greedy exploration which achieve $\Otilde(T^{2/3})$ regret.
Only recently progress has been made on offering $\Otilde(\sqrt{T})$ regret guarantees for computationally tractable algorithms.
\citet{abeille18} show that Thompson sampling achieves $\Otilde(\sqrt{T})$ (frequentist) regret for the case when the state and inputs
are both scalars. 
In a Bayesian setting \citet{ouyang17} showed that Thompson sampling achieves $\Otilde(\sqrt{T})$ \emph{expected} regret.
%\citet{faradonbeh18} argue that certainty equivalence control with $\varepsilon$-greedy achieves (ignoring log factors) $C_1 \sqrt{T}$ regret w.h.p. (c.f.~Algorithm~1 and Theorem~11) as long as $T \geq C_0$,
%but they leave the dependence of the quantities on the problem parameters un-specified. 
\citet{faradonbeh18} argue that
certainty equivalence control with an epsilon-greedy-like scheme achieves $\Otilde(\sqrt{T})$ regret, though their work does not provide any explicit dependencies on instance parameters.
Finally, \citet{cohen19} also give an efficient algorithm based on semidefinite programming that achieves $\Otilde(\sqrt{T})$ regret.
%and is based on a relaxation developed by earlier work from \citet{cohen18}. 

The literature for LQG is less complete, with most of the focus
on the estimation side. \citet{hardt2016gradient} show that gradient descent can be used
to learn a model with good predictive performance, under strong technical assumptions on the $A$ matrix.
A line of work \cite{hazan17,hazan18} has focused on using spectral filtering techniques to
learn a predictive model with low regret. 
Beyond predictive performance, several works
\cite{oymak2018non,simchowitz2019learning,tsiamis19} show how to learn the system dynamics up to a similarity
transform from input/output data. 
%As discussed earlier, these works can be used to supply the necessary 
%inputs to our main LQG result, Theorem~\ref{thm:lqg_subopt}. 
Finally, we remark that \citet{boczar18}
give sub-optimality guarantees for output-feedback of a single-input-single-output (SISO) linear system with no process noise.

A key part of our analysis involves bounding the perturbation of 
solutions to the discrete algebraic Riccati equation.
While there is a rich line of work studying perturbations of Riccati equations
\cite{konstantinov1993perturbation, konstantinov2003perturbation,
sun1998perturbation, sun1998sensitivity}, the results in the literature are
either asymptotic in nature or difficult to use and interpret.  
We clarify the operator-theoretic result of \citet{konstantinov1993perturbation} and provide an explicit upper bound on the
perturbation based on their proof strategy. Also, we take a new direct approach and
use an extended notion of controllability to give a constructive and simpler
result. 
While the result of \citet{konstantinov1993perturbation} applies more generally to systems that are 
stabilizable, we give examples of linear systems for which our new perturbation result is tighter.

%Our new perturbation result exploits a natural
%extension of the notion of \emph{controllability} of a linear system from
%control theory: we introduce a mild assumption which yields concise, interpretable perturbation bounds
%and drastically simplifies the proof strategy. Moreover, we present an
%interpretable variant of the result of \citet{konstantinov1993perturbation}.
%We note that the two guarantees we show on the sensitivity of Riccati equations are not directly
%comparable; each result yields a better guarantee than the other for some class
%of linear systems. 

%% file: conclusion.tex
%!TEX root = paper.tex

\section{Conclusion}
\label{sec:conclusion}

Though a na{\"{i}}ve Taylor expansion suggests that the fast rates we derive here must be achievable, precisely computing such rates has been open since the 80s. 
%This paper shows that fast rates are indeed achievable. 
All of the pieces we used here have existed in the literature for some time, and perhaps it has just required a bit of time to align contemporary rate-analyses in learning theory with earlier operator theoretic work in optimal control.
%
% It is moreover quite surprising that the naive $\varepsilon$-greedy strategy of adding a bit of excitation noise to nominal control achieves optimal regret. While this method is straightforward and is used pervasively in engineering practice, its analysis has remained a longstanding challenge.
%
There remain many possible extensions to this work. The robust control approach of \citet{dean2017sample} applies to many different objective functions besides quadratic costs, such as $\mathcal{H}_\infty$ and $\mathcal{L}_1$ control. It would
be interesting to know whether fast rates for control are possible for other objective functions.
%
%Another possible direction is to extend our results
%for the setting of partial observability, where we observe
%$\mathbf{y}_t = C \vecx_t$ instead of $\vecx_t$ directly.
%In this setting, some results exist on the estimation side \cite{hazan17, hazan18, oymak2018non, simchowitz2019learning},
%but we are currently unaware of any bounds on the cost suboptimality gap.
%
% Furthermore, while our results show that $\varepsilon$-greedy exploration
% is optimal for the adaptive LQR formulation from \cite{abbasi2011regret},
% it is unknown whether or not this is still the case when the regret measure
% changes to pseudo-regret.
%
Finally, determining the optimal minimax rate for both LQR and LQG
would allow us to understand the tradeoffs between nominal and robust control
at a more fine grained level.

%% file: appendix.tex
\section{Proof of Lemma \ref{lem:lipschitz_matrix_powers}}
\label{app:proof_lemma_matrix_perturb}

This proof is a simple modification of Lemma D.1 in \cite{dean2017sample}.
We replicate the argument here for completeness.

   Fix an integer
  $k \geq 1$.  Consider the expansion of $(M + \Delta)^k$ into $2^k$
  terms.  Label all these terms as $T_{i,j}$ for $i=0, ..., k$ and
  $j=1, ..., {k \choose i}$ where $i$ denotes the degree of $\Delta$
  in the term (hence there are ${k \choose i}$ terms with a degree of
  $i$ for $\Delta$).  Using the fact that
  $\norm{M^k} \leq \transient{M}{\rho} \rho^k$ for all $k \geq 0$, we can bound
  $\norm{T_{i,j}} \leq \transient{M}{\rho}^{i+1} \rho^{k-i} \norm{\Delta}^i$.  Hence by
  triangle inequality:
  \begin{align*}
    \norm{(M + \Delta)^k} &\leq \sum_{i=0}^{k} \sum_{j} \norm{T_{i,j}} \\
                          &\leq \sum_{i=0}^{k} {k \choose i} \transient{M}{\rho}^{i+1} \rho^{k-i} \norm{\Delta}^i \\
                          &= \transient{M}{\rho} \sum_{i=0}^{k} {k \choose i} (\transient{M}{\rho} \norm{\Delta})^i \rho^{k-i} \\
                          &= \transient{M}{\rho} (\transient{M}{\rho}\norm{\Delta} + \rho)^k.
  \end{align*}
  To prove the first part of the lemma we follow the same argument. We
  find
  \begin{align*}
    \norm{(M + \Delta)^k - M^k} &\leq \sum_{i=1}^{k} \sum_{j} \norm{T_{i,j}} \\
                                &\leq \sum_{i=1}^{k} {k \choose i} \transient{M}{\rho}^{i+1} \rho^{k-i} \norm{\Delta}^i \\
                                &= \transient{M}{\rho} \sum_{i=1}^{k} {k \choose i} (\transient{M}{\rho} \norm{\Delta})^i \rho^{k-i} \\
                                &= \transient{M}{\rho} \left[ (\transient{M}{\rho}\norm{\Delta} + \rho)^k - \rho^k \right] \\
                                &\leq k C^2_M (\transient{M}{\rho}\norm{\Delta} + \rho)^{k - 1} \norm{\Delta}\:,
  \end{align*}
  where the last inequality follows from the mean value theorem
  applied to the function $z \mapsto z^k$.

\section{Proof of Lemma \ref{lem:controllable_perturb}}
\label{app:proof_controllable_perturb}

We can write
\begin{align*}
  \underline{\sigma} \left( \begin{bmatrix}
    \Bhat & \Ahat \Bhat & \ldots & \Ahat^{\ell - 1} \Bhat
  \end{bmatrix} \right) = \min_{v\in \Scal^{\statedim - 1}} \left\|
  v^\top \begin{bmatrix} \Bhat & \Ahat \Bhat & \ldots & \Ahat^{\ell
      - 1} \Bhat
  \end{bmatrix} \right\| .
\end{align*}
Fix an arbitrary unit vector $v$ in $\R^\inputdim$. Then,
\begin{align*}
  &\bignorm{ v^\top \begin{bmatrix}
    \Bstar & \Astar \Bstar & \ldots & \Astar^{\ell - 1} \Bstar
  \end{bmatrix} - v^\top \begin{bmatrix} \Bhat & \Ahat \Bhat &
    \ldots & \Ahat^{\ell - 1} \Bhat
  \end{bmatrix} } \\
  &\leq \bignorm{ v^\top \begin{bmatrix} \Bstar & \Astar \Bstar & \ldots & \Astar^{\ell - 1} \Bstar \end{bmatrix} -
    v^\top \begin{bmatrix} \Bstar & \Ahat \Bstar & \ldots & \Ahat^{\ell - 1} \Bstar \end{bmatrix} } \\
      &\quad + \bignorm{ v^\top \begin{bmatrix} \Bstar & \Ahat \Bstar & \ldots & \Ahat^{\ell - 1} \Bstar \end{bmatrix}
        - v^\top \begin{bmatrix} \Bhat & \Ahat \Bhat & \ldots & \Ahat^{\ell - 1} \Bhat \end{bmatrix} } \\
  &\leq  \varepsilon \ell^{\frac{3}{2}} \transient{\Astar}{\rho}^2 \beta^{\ell - 1} \| \Bstar \| + \varepsilon \sqrt{\ell} \transient{\Astar, \rho} \beta^{\ell - 1} \\
  &\leq \varepsilon \ell^{\frac{3}{2}} \transient{\Astar}{\rho}^2 \beta^{\ell - 1} \left(\| \Bstar \| + 1 \right).
\end{align*}
We used $\ell \geq 1$, $\transient{\Astar}{\rho} \geq 1$, Lemma~\ref{lem:lipschitz_matrix_powers},
and the upper bound $\norm{M}^2 \leq \sum_{i,j} \norm{M_{i,j}}^2$ on the operator norm of a block matrix.
The conclusion follows by the triangle inequality.

%% file: proof_op_bound.tex
\section{Proof of Lemma~\ref{lem:operator_93}}
\label{app:lem_93}

We wish to upper bound $\norm{\opPhi(X)}$ and $\norm{\opPhi(X_1) -
\opPhi(X_2)}$ for $X$, $X_1$, and $X_2$ in $\calS_\nu$. First, we upper bound
the operator norm of the linear operator $\opT^{-1}$, the inverse of  $\opT\colon X \mapsto X - \Lstar^\top X \Lstar$. Since
$\Lstar$ is a stable matrix, the linear map $\opT$ must be invertible.
Moreover, when $\Lstar$ is stable and $X - \Lstar^\top X \Lstar = M$ for some
matrix $M$, we know that $X = \sum_{k = 0}^{\infty} (\Lstar^k)^\top M \Lstar^k$.
Therefore, by the triangle inequality, the operator norm of $\opT^{-1}$ can be upper bounded by
$\norm{\opT^{-1}} \leq \frac{\transient{\Lstar}{\rho}^2}{1 - \rho^2}$.
Before we proceed with the rest of the proof  we present a technical lemma which will be used several times.
\begin{mylemma}
  \label{lem:psd_norm_bound}
  Let $M$ and $N$ be two positive semidefinite matrices of the same dimension. Then $\norm{N(I + MN)^{-1}} \leq \norm{N}$.
  \end{mylemma}
\begin{proof}
  We assume that $M$ and $N$ are invertible. If they are not, we can work with the matrices $M + \nu I$ and $N + \nu I$ and take the limit of $\nu$ going to zero. Then, we have $N (I + MN)^{-1} = N N^{-1}(N^{-1} + M)^{-1} = (N^{-1} + M)^{-1} \preceq N$.
\end{proof}

Next, recall that $\opH(X) = L_\star^\top X \left(I + S_\star(\Pstar + X)\right)^{-1} S_\star X L_\star$.
Then, Lemma~\ref{lem:psd_norm_bound} yields
\begin{align*}
  \norm{\opH(X)} \leq \norm{\Lstar}^2 \norm{\Sstar} \norm{X}^2.
\end{align*}
We turn our attention to the difference $F(\Pstar + X, \Astar, \Bstar) - F(\Pstar + X, \Ahat, \Bhat)$. We use the notation $P_X$ as a shorthand for $\Pstar + X$. Then, by Equation~\ref{eq:riccati_op_short} we find
\begin{align}
  \nonumber
  F(P_X, \Ahat, \Bhat, \Qh) &- F(P_X, \Astar, \Bstar, Q_\star) = \Astar^\top P_X \left(I + \Sstar P_X \right)^{ - 1} \Astar -
                                                  \Ahat^\top P_X \left(I + \Shat P_X \right)^{ - 1} \Ahat - \Delta_Q\\
                                                &= \Astar^\top P_X (I + \Sstar P_X)^{-1} \Delta_S P_X (I + \Shat P_X)^{-1} \Astar - \Astar^\top P_X (I + \Shat P_X)^{-1} \Delta_A \nonumber \\
                                                &\quad - \Delta_A^\top P_X (I + \Shat P_X)^{-1} \Astar - \Delta_A^\top P_X (I + \Shat P_X)^{-1} \Delta_A - \Delta_Q. 
                                                  \label{eq:f_perturb}
\end{align}
Then,
\begin{align*}
  \norm{F(\Pstar + X, \Ahat, \Bhat, \Qh) &- F(\Pstar + X, \Astar, \Bstar, Q_\star)} \\
  &\leq \norm{\Astar}^2 \norm{P_X}^2 \norm{\Delta_S} + 2\norm{\Astar} \norm{P_X} \varepsilon + \norm{P_X} \varepsilon^2 + \varepsilon,
\end{align*}
where we used Lemma~\ref{lem:psd_norm_bound}. Since $X \in \calS_\nu$, we know $\norm{X} \leq \nu$ and hence $\norm{P_X} \leq \norm{\Pstar} + \nu$. We assumed that $\nu \leq 1 / 2$ and so $\norm{P_X} \leq \norm{\Pstar} + 1$. Now, we know that $\norm{\Delta_S} \leq 2\norm{\Bstar} \norm{R^{-1}} \varepsilon + \norm{R^{-1}} \varepsilon^2$ and since we assumed $\varepsilon \leq \norm{\Bstar}$, we have $\norm{\Delta_S} \leq 3 \norm{\Bstar} \norm{R^{-1}} \varepsilon$. Therefore,
\begin{align*}
\norm{\opPhi(X)} \leq \frac{\transient{\Lstar}{\rho}^2}{1 - \rho^2} \left[\norm{\Lstar}^2 \norm{\Sstar} \nu^2 + 3 \norm{\Astar}_+^2 \norm{\Pstar}^2_+ \norm{\Bstar}_+ \norm{R^{-1}}_+ \varepsilon \right].
\end{align*}

 We use Lemma~\ref{lem:psd_norm_bound}, the assumption $\nu \leq \norm{\Sstar}^{-1}$, and the definition of $\opH$ to upper bound
\begin{align*}
  \norm{\opH(X_1) - \opH(X_2)} &\leq \norm{\Lstar}^2\left[ \norm{\Sstar}^2 \nu^2 + 2\norm{\Sstar} \nu \right] \norm{X_1 - X_2} \leq 3 \norm{\Lstar}^2 \norm{\Sstar} \nu \norm{X_1 - X_2}. 
\end{align*}

Let us denote $\opG(X) = F(\Pstar + X, \Ahat, \Bhat, \Qh) - F(\Pstar + X, \Astar, \Bstar, Q_\star)$. In order to upper bound $\norm{\opG(X_1) - \opG(X_2)}$ we first upper bound the norm of $(I + \Sstar P_X)^{-1}$ and $(I + \Shat P_X)^{-1}$. Since $\norm{X} \leq \nu \leq 1/2$ and since $\Pstar \succeq I$, by Lemma~\ref{lem:psd_norm_bound} we get
\begin{align*}
\norm{(I + \Sstar P_X)^{-1}} = \norm{P_X^{-1} P_X (I + \Sstar P_X)^{-1}} \leq \norm{P_X^{-1}} \norm{P_X (I + \Sstar P_X)^{-1}} \leq 2 \norm{P_X}.
\end{align*}
Therefore, after some algebraic manipulations, we obtain
\begin{align*}                    
 \norm{\opG(X_1) - \opG(X_2)} &\leq 32 \varepsilon \norm{\Astar}_+^2 \norm{\Pstar}^3_+ \norm{\Bstar}^3_+ \norm{R^{-1}}^2_+  \norm{X_1 - X_2}.
\end{align*}

%% file: proofs_lqg.tex
%!TEX root = paper.tex

\section{Proofs for LQG}
\label{sec:proofs:lqg}

\subsection{Proof of Theorem~\ref{thm:lqg_subopt}}
The proof strategy works as follows.
Note that unlike the case for LQR, we were unable to reuse or derive an
exact second order perturbation bound analogous to 
Lemma~\ref{lem:second_order_lqr_perturbation} for LQG.
Hence we resort to an argument based on a simple Taylor expansion.

We first assume that $\Kh$ satisfies $\norm{\Kh - K_\star T^{-1}} \leq \varepsilon$ and that
$\varepsilon \leq 1$.
Let $\Theta = (A, B, C, W, V, Q, R)$ be the parameters that specific
an instance of LQG in \eqref{eq:lqg}. Here, we will be slightly more general than in
\eqref{eq:lqg} and allow the process noise $\vecw_t$ and
observation noise $\vecv_t$ to have non-isotropic covariances $W$ and $V$, respectively.
Let $J(\Ah, \Bh, \Ch, \Kh, \Lh; \Theta)$ denote the following cost:
\begin{align*}
	J(\Ah, \Bh, \Ch, \Kh, \Lh; \Theta) &= \lim_{T \to \infty} \E\left[ \frac{1}{T} \sum_{t=1}^{T} \vecy_t^\T Q \vecy_t + \vecu_t^\T R \vecu_t \right] ~~\text{s.t.} \\
	&~~~~~~~\vecx_{t+1} = A \vecx_t + B \vecu_t + \vecw_t \:, \:\: \vecw_t \sim \calN(0, W) \:,\\
	&~~~~~~~~~~~\vecy_t = C \vecx_t + \vecv_t \:, \:\: \vecv_t \sim \calN(0, V) \:,\\
	&~~~~~~~~~~~\vecu_t = \Kh \hatvecx_t \:, \\
	&~~~~~~~\hatvecx_{t+1} = \Ah \hatvecx_t + \Bh \vecu_t + \Lh (\vecy_t - \Ch \hatvecx_t) \:.
\end{align*}
As observed earlier, we have that
$J(\Ah, \Bh, \Ch, \Kh, \Lh; \Theta) = J(\Ah, \Bh, \Ch, \Kh, \Lh, \tTheta)$
for any $\tTheta$ equal to $(T A T^{-1}, T B, C T^{-1}, T W T^\T, V, Q, R)$ for some
invertible $T$.
Hence, we can assume that $T = I$ moving forward.

We consider the closely related cost $J_s(\Ah,\Bh,\Ch,\Kh,\Lh;\Theta)$ defined as:
\begin{align*}
	J_s(\Ah, \Bh, \Ch, \Kh, \Lh; \Theta) &= \lim_{T \to \infty} \E\left[ \frac{1}{T} \sum_{t=1}^{T} \vecx_t^\T Q \vecx_t + \vecu_t^\T R \vecu_t \right] ~~\text{s.t.} \\
	&~~~~~~~\vecx_{t+1} = A \vecx_t + B \vecu_t + \vecw_t \:, \:\: \vecw_t \sim \calN(0, W) \:,\\
	&~~~~~~~~~~~\vecy_t = C \vecx_t + \vecv_t \:, \:\: \vecv_t \sim \calN(0, V) \:,\\
	&~~~~~~~~~~~\vecu_t = \Kh \hatvecx_t \:, \\
	&~~~~~~~\hatvecx_{t+1} = \Ah \hatvecx_t + \Bh \vecu_t + \Lh (\vecy_t - \Ch \hatvecx_t) \:.
\end{align*}

Observe that:
\begin{align*}
	J(\Ah,\Bh,\Ch,\Kh,\Lh; \Theta) = J_s(\Ah,\Bh,\Ch,\Kh,\Lh; (A,B,C,W,V,Q_c,R)) + \Tr(Q V) \:, \:\: Q_c := C^\T Q C \:.
\end{align*}
Let $J_{s,\star}$ denote the optimal value for $J_s$.
This decomposition shows that:
\begin{align*}
	J(\Ah,\Bh,\Ch,\Kh,\Lh; \Theta) - J_\star =
	J_s(\Ah,\Bh,\Ch,\Kh,\Lh; (A,B,C,W,V,Q_c,R)) - J_{s,\star} \:.
\end{align*}
Hence we will work with $J_s$ from here on.
In the sequel, we will drop the dependence of $J$ on $\Theta$.
We write $\Xih = (\Ah,\Bh,\Ch,\Kh,\Lh)$
and $\Xi_\star = (A,B,C,K_\star,L_\star).$
By Taylor's theorem,
\begin{align*}
J_s(\Xih) - J_{s,\star} &= [D J_s(\Xi_\star)](\Delta) + \frac{1}{2} [D^2 J_s(\tXi)](\Delta,\Delta) \\
&=  \frac{1}{2} [D^2 J_s(\tXi)](\Delta,\Delta) \:.
\end{align*}
where the first equality holds by the optimality of $\Xi_\star$ and
$\tXi$ is an element along the ray between $\Xih$ and $\Xi_\star$.

We start by defining $\vece_t = \hatvecx_t - \vecx_t$.
Observe that:
\begin{align}
    \cvectwo{\vecx_t}{\hatvecx_t} = S \cvectwo{\vecx_t}{\vece_t} \:, \:\: S = \bmattwo{I}{0}{I}{I} \:. \label{eq:similarity_transform}
\end{align}
%Furthermore, we have that $S^{-1} = \bmattwo{I}{0}{-I}{I}$.
We write:
\begin{align}
    \cvectwo{\vecx_{t+1}}{\hatvecx_{t+1}} &= \cvectwo{ A \vecx_t + B \Kh \hatvecx_t + \vecw_t }{ \Ah \hatvecx_t + \Bh \Kh \hatvecx_t + \Lh( C\vecx_t + \vecv_t - \Ch \hatvecx_t  )  } \nonumber \\
    &= \bmattwo{ A }{B\Kh}{ \Lh C }{  (\Ah+\Bh\Kh) - \Lh \Ch } \cvectwo{\vecx_t}{\hatvecx_t} + \bmattwo{I}{0}{0}{\Lh} \cvectwo{\vecw_t}{\vecv_t} \nonumber \\
    &=: M(\Ah,\Bh,\Ch,\Kh,\Lh) \cvectwo{\vecx_t}{\hatvecx_t} + \bmattwo{I}{0}{0}{\Lh} \cvectwo{\vecw_t}{\vecv_t} \label{eq:lifted_update_eq} \:.
\end{align}
We abbreviate $\Mh := M(\Ah,\Bh,\Ch,\Kh,\Lh)$.
Applying a similarity transform to $\Mh$:
\begin{align*}
    S^{-1} \Mh S &= \bmattwo{I}{0}{-I}{I} \bmattwo{ A }{B\Kh}{ \Lh C }{  (\Ah+\Bh\Kh) - \Lh \Ch } \bmattwo{I}{0}{I}{I} \\
    &= \bmattwo{ A+B\Kh }{ B\Kh }{ (\Ah - A) + (\Bh - B)\Kh + \Lh (C - \Ch) }{ (\Ah - \Lh\Ch) + (\Bh-B)\Kh } \\
    &=: \Nh \:.
\end{align*}
Therefore $\Mh$ is stable iff $\Nh$ is stable.

By \eqref{eq:lifted_update_eq}, if $\Mh$ is stable, then the stationary distribution of $(\vecx_t, \hatvecx_t)$ is given by
$\calN(0, \Sigmah)$ with $\Sigmah = \dlyap\left(\Mh^\T, \bmattwo{W}{0}{0}{\Lh V \Lh^\T}\right)$.
Therefore, we have that:
\begin{align*}
	J_s(\Ah,\Bh,\Ch,\Kh,\Lh) = \Tr\left( \bmattwo{Q_c}{0}{0}{\Kh^\T R \Kh} \Sigmah \right) \:.
\end{align*}

Now for any $(\tA,\tB,\tC,\tK,\tL)$, observe we can write $\tM = M(\tA,\tB,\tC,\tK,\tL)$ 
in terms of $M_\star = M(A,B,C,K_\star,L_\star)$ as:
\begin{align*}
    \tM &= \tM - M_\star + M_\star \\
    &= \bmattwo{0}{B(\tK-K_\star)}{(\tL-L_\star)C}{(\tA-A) + (\tB-B)\tK + B(\tK-K_\star) - ((\tL-L_\star)\tC + L_\star(\tC-C))} + M_\star \\
    &=: \Delta + M_\star \:.
\end{align*}
Using the assumption that $\varepsilon \leq 1$,
we can bound $\norm{\Delta}$ by:
\begin{align*}
    \norm{\Delta} &\leq \norm{B} \varepsilon_K + \norm{C} \varepsilon_L + \varepsilon_A + (\norm{K_\star}+1)\varepsilon_B + \norm{B}\varepsilon_K + (\norm{C}+1)\varepsilon_L  + \norm{L_\star}\varepsilon_C \\
    &\leq \varepsilon_A + (\norm{K_\star}+1)\varepsilon_B + \norm{L_\star}\varepsilon_C + 2(\norm{C}+1)\varepsilon_L + 2\norm{B}\varepsilon_K \\
    &\leq 2 \Gamma_\star (\varepsilon_A + \varepsilon_B + \varepsilon_C + \varepsilon_K + \varepsilon_L) \\
    &\leq 10 \Gamma_\star \varepsilon \:.
\end{align*}
%where we defined $\Gamma_\star = 1 + \max\{ \norm{A}, \norm{B}, \norm{C}, \norm{K_\star}, \norm{L_\star} \}$.
%Recall we defined $\varepsilon = \varepsilon_A + \varepsilon_B + \varepsilon_C + \varepsilon_K + \varepsilon_L$.
Let $\rho \geq \rho(M_\star)$.
If $10 \Gamma_\star \cdot \varepsilon \leq \frac{1-\rho}{2\tau(M_\star, \rho)}$,
then we have $\rho( \tM ) \leq (1+\rho)/2$
and $\tau( \tM, (1+\rho)/2 ) \leq \tau(M_\star, \rho)$.
This holds for any $(\tA,\tB,\tC,\tK,\tL)$ along the ray
of $(\Ah,\Bh,\Ch,\Kh,\Lh)$ and $(A,B,C,K,L)$.
Therefore for any point along this ray we have that
$\tM$ is stable and that for any fixed $H$,
$\norm{\dlyap(\tM^\T, H)} \leq \frac{\norm{H} \tau(M_\star, \rho) }{1-\Avg(\rho,1)^2}$.

Next, we consider $\tXi = (\tA,\tB,\tC,\tK,\tL)$
and differentiate $J_s$ twice w.r.t.~$\tXi$, assuming that $\tXi$ corresponds
to a point along the ray $\Xih = (\Ah,\Bh,\Ch,\Kh,\Lh)$ and $\Xi_\star = (A,B,C,K_\star,L_\star)$.
By the chain rule:
\begin{align*}
    &[D^2 J_s(\tXi)](\Delta, \Delta) \\
    &= \Tr\left(\bmattwo{Q_c}{0}{0}{\tK^\T R \tK} [D^2 \Sigma(\tXi)](\Delta, \Delta)\right) + \Tr\left(\bmattwo{0}{0}{0}{\Delta_K^\T R \tK + \tK^\T R \Delta_K} [D \Sigma(\tXi)](\Delta)\right) \\
    &\qquad+ 2\Tr\left(\bmattwo{0}{0}{0}{\Delta_K^\T R \Delta_K} \Sigma(\tXi)\right) \\
    &\leq (\Tr(Q_c) + \Tr(R)(\norm{K_\star}+1)^2)\norm{ [D^2 \Sigma(\tXi)](\Delta, \Delta) } \\
    &\qquad + 2 \Tr(R)(\norm{K_\star}+1) \varepsilon_K \norm{[D\Sigma(\tXi)](\Delta)} + 2\Tr(R)\varepsilon_K^2 \norm{\Sigma(\tXi)} \\
    &\leq (\Tr(Q_c) + \Tr(R)\Gamma_\star^2) \left( \norm{ [D^2 \Sigma(\tXi)](\Delta, \Delta) } + 2 \varepsilon_K \norm{[D\Sigma(\tXi)](\Delta)} +  2\varepsilon_K^2 \norm{\Sigma(\tXi)}  \right) \:.
\end{align*}

Given the composition $h(x) = g(f(x))$, we have by the chain rule:
\begin{align}
    [D h(x)](\Delta) &= [D g(f(x))]( [Df(x)](\Delta)) \label{eq:chain_one}  \:, \\
    [D^2 h(x)](\Delta, \Delta) &= [D^2 g(f(x))](  [Df(x)](\Delta), [Df(x)](\Delta) ) + [D g(f(x))]( [D^2 f(x)](\Delta, \Delta) ) \label{eq:chain_two} \:.
\end{align}

We apply this with $g(M, Q) = \dlyap(M, Q)$ and
\begin{align*}
    f(\tXi) = \left(\bmattwo{ A }{B\tK}{ \tL C }{  (\tA+\tB\tK) - \tL \tC }^\T, \bmattwo{W}{0}{0}{\tL V \tL^\T}\right) \:.
\end{align*}
Therefore the composition $h = g \circ f$ is $h(\Xi) = \Sigma(\Xi)$.
Differentiating $f$,
\begin{align*}
    [Df(\tXi)](\Delta) &= \left(\bmattwo{0}{B\Delta_K}{\Delta_L C}{\Delta_A + \tB \Delta_K + \Delta_B \tK - (\tL \Delta_C + \Delta_L \tC)}^\T, \bmattwo{0}{0}{0}{ \tL V \Delta_L^\T + \Delta_L V \tL^\T } \right) \:, \\
    [D^2 f(\tXi)](\Delta, \Delta) &= \left( \bmattwo{0}{0}{0}{2\Delta_B \Delta_K - 2 \Delta_L \Delta_C}^\T, \bmattwo{0}{0}{0}{ 2 \Delta_L V \Delta_L^\T } \right) \:.
\end{align*}
Letting $P = g(M, Q)$, we have:
\begin{align}
    &[D g(M, Q)](\Delta) = g(M, \Delta_Q + M^\T P \Delta_M + \Delta_M^\T P M) \:, \nonumber \\
    &[D^2 g(M, Q)](\Delta, \Delta) = 2g(M, \Delta_M^\T P \Delta_M + M^\T g(M, \Delta_Q) \Delta_M + \Delta_M^\T g(M, \Delta_Q) M) \nonumber \\
    &\qquad+ 2g(M, \Delta_M^\T g(M, \Delta_M^\T P M + M^\T P \Delta_M) M + M^\T g(M, \Delta_M^\T P M + M^\T P \Delta_M) \Delta_M ) \label{eq:g_d2} \:.
\end{align}
The derivation for \eqref{eq:g_d2} is postponed for now.

Let $\calS_M(Q)$ be the linear operator $Q \mapsto g(M, Q)$.
Observe that $\norm{\calS_M} \leq \frac{\tau^2(M, \rho)}{1-\rho^2}$
and $\norm{M^\T g(M, Q)} \leq \frac{\tau^2(M, \rho)}{1-\rho^2} \norm{Q}$.
With this, we can bound,
\begin{align*}
    \norm{[D g(M,Q)](\Delta)} &\leq \frac{\tau^2(M, \rho)}{1-\rho^2} \norm{\Delta_Q} + 2\frac{\tau^4(M,\rho)}{(1-\rho^2)^2} \norm{Q}\norm{\Delta_M}  \:, \\
    \norm{[D^2 g(M,Q)](\Delta,\Delta)} &\leq 10 \frac{\tau^6(M,\rho)}{(1-\rho^2)^3} \norm{Q}\norm{\Delta_M}^2 + 4 \frac{\tau^4(M,\rho)}{(1-\rho^2)^2} \norm{\Delta_Q}\norm{\Delta_M}\:.
\end{align*}

Therefore,
\begin{align*}
    &\norm{[D g(f(\tXi))]([Df(\tXi)](\Delta))} \\
    &\leq 2\frac{\tau^2(\tM,\trho)}{1-\trho^2}\norm{\tL}\norm{V}\varepsilon_L +  20 \frac{\tau^4(\tM,\trho)}{(1-\trho^2)^2} \max\{\norm{W},\norm{\tL V \tL^\T}\} \Gamma_\star \varepsilon \\
    &\leq 2 \frac{\tau^2(M_\star,\rho)}{1-\Avg(\rho,1)^2} \Gamma_\star \norm{V}\varepsilon_L + 20 \frac{\tau^4(M_\star,\rho)}{(1-\Avg(\rho,1)^2)^2} \max\{\norm{W},\Gamma_\star^2\norm{V}\} \Gamma_\star \varepsilon \\
    &\leq 22\frac{\tau^4(M_\star,\rho)}{(1-\Avg(\rho,1)^2)^2} \max\{\norm{W},\norm{V}\} \Gamma_\star^3 \varepsilon \:.
\end{align*}
and,
\begin{align*}
    &\norm{[D g(f(\tXi))]([ D^2 f(\tXi)](\Delta, \Delta))} \\
    &\leq 2\frac{\tau^2(\tM, \trho)}{1-\trho^2} \norm{V}\varepsilon_L^2 + 4 \frac{\tau^4(\tM, \trho)}{(1-\trho^2)^2} \max\{\norm{W}, \norm{\tL V \tL^\T}\} (\varepsilon_B\varepsilon_K + \varepsilon_L\varepsilon_C) \\
    &\leq 2 \frac{\tau^2(M_\star,\rho)}{1-\Avg(\rho,1)^2} \norm{V}\varepsilon_L^2 + 8 \frac{\tau^4(M_\star,\rho)}{(1-\Avg(\rho,1)^2)^2} \max\{\norm{W}, \Gamma_\star^2\norm{V}\}  \varepsilon^2 \\
    &\leq 10 \frac{\tau^4(M_\star,\rho)}{(1-\Avg(\rho,1)^2)^2} \max\{\norm{W},\norm{V}\} \Gamma_\star^2 \varepsilon^2 \:.
\end{align*}
and also,
\begin{align*}
    &\norm{ [D^2 g(f(\tXi))]( [Df(\tXi)](\Delta),  [Df(\tXi)](\Delta)   ) } \\
    &\leq 10 \frac{\tau^6(\tM, \trho)}{(1-\trho^2)^3} \max\{ \norm{W}, \norm{\tL V \tL^\T} \} (10 \Gamma_\star \varepsilon)^2 + 4 \frac{\tau^4(\tM, \trho)}{(1-\trho^2)^2}(2 \norm{\tL}\norm{V} \varepsilon_L)  (10\Gamma_\star \varepsilon) \\
    &\leq 1000 \frac{\tau^6(M_\star,\rho)}{(1-\Avg(\rho,1)^2)^3} \max\{\norm{W},\norm{V}\} \Gamma_\star^4 \varepsilon^2 + 80 \frac{\tau^4(M_\star,\rho)}{(1-\Avg(\rho,1)^2)^2} \norm{V}\Gamma_\star \varepsilon^2 \\
    &\leq 1080 \frac{\tau^6(M_\star,\rho)}{(1-\Avg(\rho,1)^2)^3} \max\{\norm{W},\norm{V}\} \Gamma_\star^4 \varepsilon^2 \:.
\end{align*}
Therefore,
\begin{align*}
    \norm{\Sigma(\tXi)} &\leq \frac{\tau^2(M_\star,\rho)}{1-\Avg(\rho,1)^2} \max\{\norm{W},\norm{V}\} \Gamma_\star^2  \:, \\
    \norm{[D\Sigma(\tXi)](\Delta)} &\leq 22\frac{\tau^4(M_\star,\rho)}{(1-\Avg(\rho,1)^2)^2} \max\{\norm{W},\norm{V}\} \Gamma_\star^3 \varepsilon \:, \\
    \norm{[D^2 \Sigma(\tXi)](\Delta,\Delta)} &\leq 1090\frac{\tau^6(M_\star,\rho)}{(1-\Avg(\rho,1)^2)^3} \max\{\norm{W},\norm{V}\} \Gamma_\star^4 \varepsilon^2  \:.
\end{align*}
Combining these bounds with the calculations for $J_s$,
\begin{align*}
    &[D^2 J_s(\tXi)](\Delta,\Delta) \\
    &\leq (\Tr(Q_c) + \Tr(R)\Gamma_\star^2) \left( \norm{ [D^2 \Sigma(\tXi)](\Delta, \Delta) } + 2\varepsilon_K \norm{[D\Sigma(\tXi)](\Delta)} +  2\varepsilon_K^2 \norm{\Sigma(\tXi)}  \right) \\
    &\leq 1136 (\Tr(Q_c) + \Tr(R)\Gamma_\star^2) \frac{\tau^6(M_\star,\rho)}{(1-\Avg(\rho,1)^2)^3} \max\{\norm{W},\norm{V}\} \Gamma_\star^4 \varepsilon^2 \\
    &\leq 1136 (\Tr(Q_c)+\Tr(R)) \frac{\tau^6(M_\star,\rho)}{(1-\Avg(\rho,1)^2)^3} \max\{\norm{W},\norm{V}\} \Gamma_\star^6 \varepsilon^2 \\
    &\leq 1136 \cdot 64 (\Tr(Q_c)+\Tr(R)) \frac{\tau^6(M_\star,\rho)}{(1-\rho^2)^3} \max\{\norm{W},\norm{V}\} \Gamma_\star^6 \varepsilon^2 \:.
\end{align*}

To finish the first part of proof, we show the derivation for \eqref{eq:g_d2}.
Observe that for $E$ small:
\begin{align*}
    (A+E)^{-1} &= (A(I+A^{-1}E))^{-1} = (I+A^{-1} E)^{-1} A^{-1} = \sum_{k=0}^{\infty} (- A^{-1} E)^{k} A^{-1} \\
    &= A^{-1} - A^{-1} E A^{-1} + A^{-1} E A^{-1} E A^{-1} + O(\norm{E}^3) \:.
\end{align*}
Therefore, letting $S = I - M^\T \otimes M^\T$,
\begin{align*}
    &(I - (M+\Delta_M)^\T \otimes (M+\Delta_M)^\T)^{-1} \vec(Q + \Delta_Q) \\
    &= S^{-1} \vec(Q + \Delta_Q) \\
    &\qquad+ S^{-1}( \Delta_M^\T \otimes M^\T + M^\T \otimes \Delta_M^\T + \Delta_M^\T \otimes \Delta_M^\T )S^{-1} \vec(Q+\Delta_Q) \\
    &\qquad+ S^{-1}( \Delta_M^\T \otimes M^\T + M^\T \otimes \Delta_M^\T )S^{-1}( \Delta_M^\T \otimes M^\T + M^\T \otimes \Delta_M^\T )S^{-1} \vec(Q+\Delta_Q) + O(\norm{\Delta}^3) \\
    &= S^{-1} \vec(Q) + S^{-1} \vec(\Delta_Q) \\
    &\qquad + S^{-1} (\Delta_M^\T \otimes M^\T + M^\T \otimes \Delta_M^\T) S^{-1} \vec(Q) + S^{-1}  (\Delta_M^\T \otimes \Delta_M^\T) S^{-1} \vec(Q) \\
    &\qquad + S^{-1} (\Delta_M^\T \otimes M^\T + M^\T \otimes \Delta_M^\T) S^{-1} \vec(\Delta_Q) \\
    &\qquad + S^{-1}( \Delta_M^\T \otimes M^\T + M^\T \otimes \Delta_M^\T )S^{-1}( \Delta_M^\T \otimes M^\T + M^\T \otimes \Delta_M^\T )S^{-1} \vec(Q) + O(\norm{\Delta}^3) \:.
\end{align*}

Now to finish the proof, we need to obtain a $\overline{\varepsilon}$
such that
\begin{align*}
  \max\{\norm{\Ah - T A_\star T^{-1}},\norm{\Bh - T B_\star},\norm{\Ch - C_\star T^{-1}},\norm{\Lh- T L_\star},\norm{\Kh - K_\star T^{-1}} \} \leq \overline{\varepsilon} \:.
\end{align*}
We do this by observing that:
\begin{align*}
	K_\star T^{-1} = \mathsf{LQR}(T \Astar T^{-1}, T \Bstar, T^{-\T} \Cstar^\T Q \Cstar T^{-1}, R) \:.
\end{align*}
Hence with $\Kh = \mathsf{LQR}(\Ah, \Bh, \Ch^\T Q \Ch, R)$, we can use the Riccati
perturbation 
between $\overline{P}_\star = \mathsf{dare}(T \Astar T^{-1}, T \Bstar, T^{-\T} \Cstar^\T Q \Cstar T^{-1}, R)$
and $\Ph = \mathsf{dare}(\Ah, \Bh, \Ch^\T Q \Ch, R)$.
We can bound, assuming that that $\varepsilon_C \leq \norm{\Cstar}$:
\begin{align*}
	\norm{ T^{-\T} \Cstar^\T Q \Cstar T^{-1} - \Ch^\T Q \Ch } \leq (2\norm{\Cstar} + \varepsilon_C)\norm{Q}\varepsilon_C \leq 3 \norm{\Cstar}\norm{Q}\varepsilon_C \:.
\end{align*}
Therefore by our hypothesis and the assumption that $f$ is monotonically increasing, 
\begin{align*}
	\norm{\Ph - \overline{P}_\star} \leq f(\max\{\varepsilon_A, \varepsilon_B, 3\norm{\Cstar}\norm{Q}\varepsilon_C \}) \leq f( 3 \norm{\Cstar}_+\norm{Q}_+ \max\{\varepsilon_A,\varepsilon_B,\varepsilon_C\}) \:.
\end{align*}

as long as $ 3 \norm{\Cstar}_+\norm{Q}_+ \max\{\varepsilon_A,\varepsilon_B,\varepsilon_C\}\leq \gamma_0$.
By Lemma~\ref{lem:K_perturb_bound}, this allows us to bound:
\begin{align*}
	\norm{\Kh - K_\star T^{-1}} &\leq \frac{7 \Gamma_\star^3}{\smin(R)} \max\{\varepsilon_A, \varepsilon_B, f( 3 \norm{\Cstar}_+\norm{Q}_+ \max\{\varepsilon_A,\varepsilon_B,\varepsilon_C\}) \} \\
	&\leq  \frac{7 \Gamma_\star^3}{\smin(R)} f( 3 \norm{\Cstar}_+\norm{Q}_+ \max\{\varepsilon_A,\varepsilon_B,\varepsilon_C\}) \:.
\end{align*}
We can set $\overline{\varepsilon} = \frac{7 \Gamma_\star^3}{\smin(R)} f( 3 \norm{\Cstar}_+\norm{Q}_+ \max\{\varepsilon_A,\varepsilon_B,\varepsilon_C\})$
because we assume that $f(\gamma) \geq \gamma$ and $\smin(R) \geq 1$.

We now complete the proof by observing that for
$M_\star = \bmattwo{\Astar}{\Bstar\Kstar}{\Lstar\Cstar}{\Astar+\Bstar\Kstar-\Lstar\Cstar}$ and
$N_\star = \bmattwo{\Astar+\Bstar\Kstar}{\Bstar\Kstar}{0}{\Astar - \Lstar\Cstar}$,
using $S$ defined in \eqref{eq:similarity_transform}, we have that
$N_\star = S^{-1} M_\star S$.
This identity has several consequences.
First, we have $\rho(N_\star) = \rho(M_\star)$.
Second, for any integer $k \geq 1$ we have
$S^{-1} M_\star^k S = N_\star^k $, and therefore
% \begin{align*}
%     S^{-1} M_\star^k S = N_\star^k = \bmattwo{(\Astar+\Bstar\Kstar)^k}{\star}{0}{(\Astar-\Lstar\Cstar)^k} \:.
% \end{align*}
\begin{align*}
    \norm{N_\star^k} \leq \mathrm{cond}(S) \norm{M_\star^k} = \frac{3+\sqrt{5}}{2} \norm{M_\star^k} \:.
\end{align*}
A nearly identical argument shows that $\norm{M_\star^k} \leq \frac{3+\sqrt{5}}{2} \norm{N_\star^k}$.
Hence, $\tau(M_\star, \rho) \asymp \tau(N_\star, \rho)$ for any $\rho$.